\documentclass{article}
\usepackage[utf8]{inputenc}
\usepackage{amsthm}
\usepackage{amsmath}
\usepackage{amsfonts}
\usepackage{latexsym}
\usepackage{geometry}
\usepackage{enumerate}
\usepackage{csquotes}
\usepackage{graphicx}
\usepackage{comment}
\usepackage{appendix}
\usepackage{hyperref}
\usepackage{bm}

\emergencystretch=\maxdimen
\def \W {\mathcal{W}}
\def \bW {\overline{\mathcal{W}}}
\def \N {\mathcal{N}}

\def \bN {\overline{\mathcal{N}}}
\def \dist {{\rm dist}}
\DeclareMathOperator{\Rm}{Rm}
\DeclareMathOperator{\Ric}{Ric}

\def \Rmin {R_{\operatorname{min}}}
\makeatletter
\newcommand*{\rom}[1]{\rm {\expandafter\@slowromancap\romannumeral #1@}}
\makeatother

\numberwithin{equation}{section}
\newtheorem{Theorem}{Theorem}[section]
\newtheorem{Proposition}[Theorem]{Proposition}
\newtheorem{Lemma}[Theorem]{Lemma}
\newtheorem{Corollary}[Theorem]{Corollary}

\theoremstyle{definition}

\title{A local Sobolev inequality on Ricci flow and its applications}
\author{Pak-Yeung Chan, Zilu Ma, Yongjia Zhang}

\begin{document}

\maketitle
\begin{abstract}
    In this article, we prove a local Sobolev inequality for complete Ricci flows. Our main result is that the local $\nu$-functional of a disk on a Ricci flow depends only on the Nash entropy based at the center of the disk, and consequently depends only on the volume of the disk. Furthermore, we introduce some applications of this local Sobolev inequality. These applications reveal the way in which the local geometry evolves along Ricci flow. In particular, we show that several classical theorems related to Perelman's monotonicity formula can be derived from our results.
\end{abstract}
\setcounter{tocdepth}{1}
\tableofcontents

\section{Introduction}

\subsection{Local Sobolev inequalities on Ricci flow}

As interpreted by Q.~Zhang \cite{Zhq07, Zhq10}, Perelman's entropy functional \cite{Per02} is a family of logarithmic Sobolev inequalities in disguise. By applying Perelman's monotonicity formula, Zhang \cite{Zhq07} proved a uniform Sobolev inequality for compact Ricci flows. The main idea of his proof, intuitively speaking, is that Perelman's $\nu$-functional (which is also a Sobolev constant; see the definition in Section 2) is monotonically increasing along a Ricci flow. Hence, for any compact Ricci flow $(M^n,g_t)_{t\in[0,T)}$, we have
$$\nu(g_t)\geq \nu(g_0)\quad\text{ for all }\quad t\in [0,T),$$
and consequently the Sobolev constant on each time-slice is determined by the geometry of the initial metric $g_0$.

This idea has also been adapted by the authors to prove a uniform Sobolev inequality for ancient Ricci flows. Recall that an ancient Ricci flow ``starts'' from $t=-\infty$, and hence there is no initial metric to begin with. We proved that, however, by improving Bamler's entropy estimates \cite{Bam20a} on ancient Ricci flows, Perelman's asymptotic shrinking gradient Ricci soliton \cite[Proposition 11.2]{Per02} works equally well as an ``initial manifold'', with which we may apply Q.~Zhang's argument; this leads to an initial version of the uniform Sobolev inequality for ancient Ricci flows \cite{CMZ21a}, where we assumed that the ancient Ricci flow in question has an asymptotic shrinking gradient Ricci soliton in the sense of Perelman \cite[Proposition 11.2]{Per02}.

Later, we improved this uniform Sobolev inequality in \cite{CMZ21b}. We found that one does not need to assume the existence of Perelman's asymptotic shrinking gradient Ricci soliton, since, as is proved by Bamler \cite{Bam20c}, every ancient Ricci flow with bounded Nash entropy has a tangent flow at infinity, which is a metric soliton regular almost everywhere. It turns out that this metric soliton can also serve as an ``initial manifold'' for the purpose of applying Q.~Zhang's argument. Hence, the main result in \cite{CMZ21b} says that every ancient Ricci flow with bounded Nash entropy has uniformly bounded $\nu$-functional, and therefore on it holds a uniform Sobolev inequality. With a different approach, Li-Wang \cite{LW20} also proved a uniform Sobolev inequality on shrinking gradient Ricci solitons, which are special ancient Ricci flows; their result is nice in that it does not assume any curvature condition on the shrinking gradient Ricci soliton.

All the results mentioned above are global Sobolev inequalities on the whole manifold, and their validity requires some explicit or implicit assumptions on the geometry of the Ricci flow. Q.~Zhang's theorem assumes the manifold to be compact. Li-Wang's result holds only on shrinking gradient Ricci solitons. Our Sobolev inequality requires the ancient solution to have time-wise bounded curvature and uniformly bounded Nash entropy; these two assumptions also imply that the ancient solution in question has a uniform lower bound for unit balls on each time-slice. However, in the most general case, the existence of a Ricci flow is guaranteed by the boundedness of curvature of the initial manifold alone (c.f. \cite{Sh89}), and the curvature remains bounded until the first singular time (c.f. \cite{Ham93}). Therefore, it makes much sense to consider a complete Ricci flow with bounded curvature on compact time intervals, without making any further geometric assumptions, such as a lower bound for the volume of unit balls. The first main theorem in this article provides some analytic tools for such a Ricci flow---we show that the local Sobolev constant depends only on the bound of the Nash entropy or on the local volume ratio. 

Most of the definitions in the statements of results  can be found in Section 2 below. Throughout this paper, we shall always consider a complete Ricci flow $(M^n,g_t)_{t\in I}$, where $I$ is an interval, with \emph{bounded curvature on each compact time interval}; this assumption is made purely for the technical convenience. We denote by $B_t(x,r)$, where $t\in I$, the ball centered at $x\in M$, of radius $r>0$, and with respect to the metric $g_t$, and by $|\,\cdot\,|_t$ the volume of a set measured with $g_t$.

\begin{Theorem}
\label{thm: nu ge nash}
Assume that $[-r^2,0]\subseteq I$.
Then for any point $x_0\in M$ and any $\tau$, $A>0,$ we have
\begin{equation}\label{localnu}
    \nu\left(B_0(x_0,Ar), g_0, \tau r^2\right)
    \ge \N_{x_0,0}(r^2) - \sqrt{n}A-\frac{n}{2}\tau - \frac{n}{2}\log(1+\tau).
\end{equation}
\end{Theorem}

Applying a standard argument (c.f. \cite{Zhq10, Ye15}), we may rewrite the above local bound of Perelman's $\nu$-functional as the following local Sobolev inequality.

\begin{Corollary}\label{Corollary}
Under the same assumptions as Theorem \ref{thm: nu ge nash}, we have that for any $n\ge 3$, there is a dimensional constant $c=c(n)$, such that for any $u\in W^{1,2}_0(B_0(x_0,Ar))$, it holds that
\begin{align}\label{Sobolevineq}
\left(\int_{B_0(x_0,Ar)}|u|^{\frac{2n}{n-2}}\,dg_0\right)^{\frac{n-2}{n}}
  \leq c e^{-\frac{2\nu_0}{n}-c\tau r^2R_{\operatorname{min}}} \int_{B_0(x_0,Ar)} \bigg(4|\nabla u|^2
  +\left( R_{g_0}-R_{\operatorname{min}}+\frac{c}{\tau r^2}\right)u^2 \bigg)\,dg_0, 
\end{align}
where $\nu_0=\N_{x_0,0}(r^2) - \sqrt{n}A-\frac{n}{2}\tau - \frac{n}{2}\log(1+\tau)$ and $R_{\operatorname{min}}=\min\left\{\inf_{B_0(x_0,Ar)} R_{g_0},0\right\}\geq-\frac{n}{2r^2}$.
\end{Corollary}

Previously, Q.~Zhang \cite[Theorem 6.3.2]{Zhq10} also proved a local Sobolev inequality assuming local curvature and volume bounds. We shall show that Zhang's theorem is actually a special case of our local Sobolev inequality above.

\begin{Corollary}[Q.~Zhang's local Sobolev inequality {\cite[Theorem 6.3.2]{Zhq10}}]\label{Qzhang}
Assume that $[0,r^2]\subset I$. Let $x_0\in M$ be a point and $A>0$ be a positive number satisfying 
\begin{gather}\label{localgeometrybounds}
|B_0(x_0,r)|_0\geq A^{-1}r^n,
\\\nonumber
|{\Ric}|\leq \frac{1}{(nr)^2}\quad\text{ on }\quad B_0(x_0,r)\times[0,r^2].
\end{gather}
Then we have 
$$\nu\left(B_{r^2}(x_0,Ar),g_{r^2},A^2r^2\right)\geq-C(n,A).$$
\end{Corollary}

Bamler \cite[Theorem 8.1]{Bam20a} proved that the Nash entropy based at a point can be estimated by the local volume ratio. Therefore, we can rewrite Theorem \ref{thm: nu ge nash} and Corollary \ref{Corollary} as follows.

\begin{Corollary}\label{Corollary2}
Assume that $[-r^2,0]\subseteq I$. Then for any positive numbers $\alpha$, $\tau$, and $A$, the following holds. For any $x_0\in M$ satisfying
\[
    |B_{0}(x_0,r)|_{0}
    \ge \alpha r^n,
\]
we have
\[
    \nu(B_{0}(x_0,A r),g_0, \tau r^2)\ge \log\alpha - 4n(A+\tau)-C(n).
\]
Furthermore, (\ref{Sobolevineq}) also holds with $\nu_0=\log\alpha - 4n(A+\tau)-C(n)$ instead. As a standard consequence, the ball $B_0(x_0,Ar)$ is strongly noncollapsed at scale $Ar$, i.e., there is a constant $\kappa=\kappa(n,\alpha,A)>0,$ such that for any $\rho\in(0,Ar]$ and $y_0\in B_{0}(x_0,Ar)$, it holds that
\[
    \sup_{B_0(y_0,\rho)} R_{g_0}\le \rho^{-2} \implies |B_0(y_0,\rho)|_{0} \ge \kappa \rho^n.
\]
\end{Corollary}

\subsection{The ``reaction-diffusion'' property of the Ricci flow}

The Ricci flow is a weakly parabolic quasilinear partial differential equation, which is sometimes known as a reaction-diffusion system in other fields of study. In a pure diffusion process, Fick's first law of diffusion says that the flux of a chemical substance always goes from regions of high concentration to regions of low concentration, and eventually, the substance tends to be distributed evenly everywhere. Such phenomenon can also be observed in the Ricci flow, especially when the initial metric is good enough. In this respect, see, for instance, \cite{BW08,BS09,Ham82,Ham86}.

In general, when the Ricci flow is noncompact without good initial condition, it is less clear how the reaction-diffusion mechanism operates. One may simply regard the Ricci flow as a background and consider the heat equation or the conjugate heat equation. For instance, the (conjugate) heat kernel coupled with the Ricci flow has some Gaussian upper and lower estimates \cite[Theorem 26.25, Theorem 26.31]{RFV3}. However, these estimates are rather coarse, and they depend not so much on the parabolicity of the Ricci flow itself than on that of the (conjugate) heat equation. What, then, should we expect from the Ricci flow as a reaction-diffusion system? Whither is the ``geometric information'' diffused?

A hint of the answer can already be seen from the works of Perelman \cite{Per02,Per03}. For instance, on a noncollapsed ancient solution with nonnegative curvature operator, one can always observe a Ricci shrinker arising along the path of $\ell$-centers, namely, the points where Perelman's $\ell$-function does not exceed $\frac{n}{2}$; see \cite[Definition 2.6]{CMZ21a}. Furthermore, Perelman proved the noncollapsing property at a point $(x,t)$ after surgery time by estimating the geometry near its $\ell$-center $(z,s)$ before surgery time. If the geometry near $(z,s)$ can be controlled, then the geometry near $(x,t)$ can be controlled accordingly. These facts indicate that the ``geometric information'' is diffused along the path of $\ell$-centers.

Recently, by a delicate estimate on the conjugate heat kernel, Bamler \cite{Bam20a} discovered a property of the Ricci flow called $H_n$-concentration---a conjugate heat kernel on a Ricci flow is centered around some points, called $H_n$-centers (\cite[Definition 3.10]{Bam20a}), in the same way an Euclidean Gaussian heat kernel is centered around its base point. However, in contrast to the case of heat equation on the Euclidean space, an $H_n$-center is not a fixed point in space. In a  pure diffusion process, if the initial distribution of a chemical substance is a Dirac delta function at the origin, then the substance is diffused forward in time and outward in space, or in other words, loosely speaking, along the family of parabolic neighborhoods of the form $B(O,r)\times[0,r^2]$, $r>0$. Analogous to the diffusion process in the Euclidean space, the Ricci flow bears a similar property along parabolic neighborhoods constructed around $H_n$-centers, called $P^*$-parabolic neighborhoods. This can be seen from Bamler's Harnack inequality of the Nash entropy (c.f. Theorem \ref{Harnack}); this result indicates that the difference of Nash entropy based at two different points is determined by the ``$P^*$-distance'' between the base points. 

As 
supported by some previous results (\cite[Theorem 6.1, Theorem 6.2, Theorem 8.1]{Bam20a}, \cite[Theorem 1.8]{CMZ21a}), the Nash entropy is closely related to the local geometry. We shall come back to this point later, and it suffices us now to view it in combination with the observation in the previous paragraph: as we shall see in the following corollary, the geometric information around one point in space-time determines the geometric information in $P^*$-parabolic neighborhoods centered at this point.

\begin{Corollary}\label{Corollary3}
Assume that $[-r^2,0]\subseteq I$. Let $x_0\in M$ and $A$, $\tau>0$. Then for any $(y,s)\in P^*(x_0,0\,|\, Ar,0,A^2r^2)$, we have
\[
    \nu(B_{s}(y,A r),g_s, \tau r^2)\ge 
    \N_{x_0,0}(r^2) - 2\sqrt{n}A-\frac{n}{2}\tau - \frac{n}{2}\log(1+\tau)-\frac{n}{2}\log(1+A^2).
\]
As a consequence, there is a constant $\kappa=\kappa(n,A)>0$ such that for any ball $B:=B_s(y,\rho)$  with $(y,s)\in P^*(x_0,0\,|\,Ar,0,A^2r^2)$ and $\rho\in(0, Ar]$, we have
\[
    \sup_B R_{g_s}\le \rho^{-2} \implies |B|_{s} \ge \kappa \rho^n.
\]
\end{Corollary}

In spite of the formal nice expression of the above result, however, without any further geometric assumption, the shape of a $P^*$-parabolic neighborhood could be very complicated. In fact, given an arbitrary space-time point $(y,s)\in M\times I$ and an arbitrary time $t\in(s,\infty)\cap I$, one may not be able to find a point $x\in M$, such that $(y,s)$ is (or is at least close to) an $H_n$-center of $(x,t)$. This is due to the property of parabolic systems (i.e. one may not solve a parabolic equation backwards, etc.). This observation is supported by Lai's recent examples \cite{Lai20}. Taking a higher-dimensional noncollapsed Lai's steady soliton, a point $(y,s)$ on the edge, where $s\ll-1$, can never be close to an $H_n$-center of $(x,0)$, where $x$ is any point. This is because along the $H_n$-centers of a fixed point, one always sees an asymptotic shrinking gradient Ricci soliton forming (c.f. \cite{CMZ21a}), which can never arise from the edges of such examples.
In fact, Perelman's reduced distance has a similar problem, namely, given a point in space-time,  this point may not be an $\ell$-center of any point on a later slice.

Another perspective from which we may understand the ``reaction-diffusion'' property of the Ricci flow along $P^*$-parabolic neighborhoods is the following Harnack inequality of the heat equation. Recall that a Harnack inequality of a parabolic equation provides some backward estimate of a heat-type equation. Some results in this respects are \cite{Ca08,Ham93a, LY86}, \cite[\S9]{Per02}, \cite[\S6]{W18}. Interestingly, the differential Harnack inequalities of Cao \cite{Ca08} and Perelman \cite[\S9]{Per02} do not depend on any local geometric condition. Nonetheless, if one wants to apply these differential Harnack inequalities to obtain some backward control, namely, to estimate the earlier slices of a solution using its later slices, then some local geometric assumption is indispensable.

Wang's Harnack inequality \cite[Proposition 6.6]{W18} shows that, for a positive supersolution to the heat equation coupled with Ricci flow, the local bound at a earlier time-slice can be controlled by that of a later slice. This Harnack inequality depends only on a very weak assumption of the Ricci curvature. In fact, inspired by Wang's result, since, as we shall see again later, the geometric information is diffused along the path of $H_n$-centers, and since the local geometric assumption is nothing but that which controls the position of an $H_n$-center, it is reasonable to hope for a Harnack inequality of heat equation along Ricci flow in terms of $P^*$-parabolic neighborhoods. We leave it to the reader to check that, from our theorem below, one can derive B.~Wang's Harnack inequality, or at least a qualitatively similar one; the method is not different from the proof of Corollary \ref{BWang}, which we do present in the sequel.

\begin{Theorem}[A Harnack inequality]\label{HeatHarnack}
Let $[-r^2,0]\subset I$ and let $H:M\times [-r^2,0]$ be a positive supersolution to the heat equation coupled with Ricci flow, namely, $\Box_t H\geq 0$. Let $x_0\in M$ be any point and $(z,-r^2)$ be an $H_n$-center of $(x_0,0)$. Then we have 
\begin{equation*}
    \Phi_0(\Lambda,A)\cdot\min_{\bar B_{-r^2}(z,Ar)}H(\cdot,-r^2)\leq \min_{\bar B_{0}(x_0,\Lambda r)}H(\cdot,0),
\end{equation*}
for all $A>\sqrt{2H_n}+4\sqrt{2\pi}$ and $\Lambda>0$, where
$$\Phi_0(\Lambda,A):=\Phi\left(\tfrac{1}{\sqrt{2}}A-\Lambda-\sqrt{H_n}-4\sqrt{\pi}\right),$$
and $\Phi$ is the function defined in \cite[\S4]{Bam20a} satisfying
$$\Phi'(t)=(4\pi)^{-\frac{1}{2}}e^{-\frac{t^2}{4}},\quad \lim_{t\to-\infty}\Phi(t)=0,\quad \lim_{t\to+\infty}\Phi(t)=1.$$
Here (and throughout the paper) $H_n:=\frac{(n-1)\pi^2}{2}+4$.
\end{Theorem}

\subsection{Local monotonicity and Nash entropy}

 Because of the interdependence of the Nash entropy and the geometry near the base point (\cite[Theorem 6.1, Theorem 8.1]{Bam20a}), and because of the interdependence of the Nash entropy and the geometry near the $H_n$-center (\cite[Theorem 6.2]{Bam20a}, \cite[Theorem 1.8]{CMZ21a}), we know that the geometry near a point in space-time depends on the geometry near its $H_n$-center, and the Nash entropy is that which relates the geometry near these two points; this, indeed, is the philosophy underlying Theorem \ref{thm: nu ge nash}. However, given this fact, we can directly estimate some geometric constants near a point in terms of what we know about its $H_n$-center---instead of going through the Nash entropy. This leads to a local monotonicity theorem.

Some examples in this respect are  \cite[\S 5]{W18} and \cite[\S 2.2]{TZ21}. These results all show how local geometry on later time-slices is dependent on that of earlier time-slices. However, because the property of $H_n$-concentration was not well understood by the time of their publication, these results all make local geometric assumptions in one way or another. In retrospective,  these local geometric assumptions can be regarded as means of controlling the possible position of an $H_n$-center. One can see this point more clearly when one views the proof of Theorem \ref{Theorem2} below (see also \cite{J21}). For other local monotonicity formulas along geometric flows, see, for example, \cite{E01, EKNT08}.

Our local monotonicity formula, on the contrary, does not depend on any local geometric assumptions, since we are making use of the properties of $H_n$-centers observed above. In particular, we show that the local $\mu$-functional around a space-time point is solely determined by the local $\mu$-functional around its $H_n$-center. This result is in the spirit of, though slightly stronger than, \cite[Theorem 5.4]{W18}.

\begin{Theorem}
\label{prop: almost mono}
Assume that $[-r^2,0]\subseteq I$. Then for any $x_0\in M$, any $H_n$-center $(z,-r^2)$ of $(x_0,0)$, any $\tau>0$, and any $A\geq 16$, we have
\[
    \mu\left(B_{{-r^2}}\left(z,2A\sqrt{H_n}r\right), g_{-r^2},(1+\tau)r^2\right)
    \le \mu\left(B_0\left(x_0,A\sqrt{H_n}r\right), g_0, \tau r^2\right)
    +  \frac{C_n}{A^2}e^{-\frac{A^2}{20}},
\]
where $C_n$ is a positive dimensional constant. 
\end{Theorem}

The sharpness of the above result can be seen from the fact that, taking $A\to \infty$, it reduces to the classical monotonicity of Perelman's $\mu$-functional \cite{Per02}.

\begin{Corollary}[Monotonicity of Perelman's $\mu$-functional]
Assume that $[-r^2,0]\subset I$. Then, for all $\tau>0$, we have
$$\mu(g_{-r^2},(1+\tau)r^2)\leq\mu(g_0,\tau r^2).$$
\end{Corollary}

B.~Wang's local monotonicity formula can also be derived from Theorem \ref{prop: almost mono}. It is worth noting that, in the corollary below, our error term $\frac{C_n}{A^2}e^{-c_nA^2}$ is better than Wang's, and neither is there any restriction on the positive scale $\tau$.

\begin{Corollary}[Local monotonicity formula of Wang {\cite[Theorem 5.4]{W18}}]\label{BWang}
Assume that $[0,T]\subset I$. Let $A\ge 1000n$ be a large constant and let $x_0\in M$ be a fixed point satisfying 
\begin{align*}
    \Ric_{g_t}(x)\leq 
    A/t
    \quad \text{ for all }\quad t\in (0,T]\ \text{ and }\ x\in B_t\left(x_0,\sqrt{t}\right).
\end{align*}
Then, for any $\tau>0$, we have
\begin{align*}
    \mu\left(g_T,B_T\left(x_0,8A\sqrt{T}\right),\tau\right)-\mu\left(g_0,B_0\left(x_0,20A\sqrt{T}\right),\tau+T\right)\ge -\frac{C_n}{A^2}e^{-c_nA^2},
\end{align*}
where $C_n<\infty$ and $c_n>0$ are both dimensional constants.
\end{Corollary}

Next, we improve our previous result \cite[Theorem 1.8]{CMZ21a}, which says that the Nash entropy depends on the local volume lower bound at an $H_n$-center. Seeing that it  determines not only the Nash entropy (c.f. \cite[Theorem 8.1]{Bam20a}), but also the local $\nu$-functional (c.f. Corollary \ref{Corollary2}), the local volume lower bound is a relatively stronger assumption. We shall weaken the assumption by showing that the Nash entropy depends only on the local $\mu$-functional near an $H_n$-center of the base point.

\begin{Theorem}
\label{thm: Nash depends on nu}
Assume that $[-r^2,0]\subseteq I$. Furthermore, assume that $R_{g_{-r^2}}\geq \Rmin$. Then, for any $x_0\in M$, any $H_n$-center $(z,-r^2)$ of $(x_0,0)$, and any $A\ge 8$, we have
\begin{align}\label{eq:nash depend on mu}
     \mu\left(B_{{-r^2}}\left(z,2A\sqrt{H_n}r\right), g_{-r^2},r^2\right)
    \le \N_{x_0,0}(r^2)
    +  C(n,\Rmin r^2,A),
\end{align}
where
$$C(n,\Rmin r^2,A)=\tfrac{C_n}{A^2}e^{-\frac{A^2}{20}}+8\left(e^{-\frac{A^2}{20}}\cdot (n-2\Rmin r^2)+e^{-\frac{A^2}{40}}\cdot(n-2\Rmin r^2)^{\frac{1}{2}}\right),$$
and $C_n$ is a dimensional constant.
\end{Theorem}

The sharpness of the above theorem can be seen from the fact that, if we take $A\to \infty$, then (\ref{eq:nash depend on mu}) reduces to a consequence of (\ref{Nash greater than Pen})
\begin{align*}
    \N_{x_0,0}(r^2)\geq \W_{x_0,0}(r^2)\geq \mu(g_{-r^2},r^2).
\end{align*}
It can be easily observed that \cite[Theorem 1.8]{CMZ21a} becomes a corollary of Theorem \ref{thm: Nash depends on nu} and Corollary \ref{Corollary2}; this shows that Theorem \ref{thm: Nash depends on nu} is indeed stronger than \cite[Theorem 1.8]{CMZ21a}.   Furthermore, later we shall see that Perelman's pseudolocality theorem \cite[Theorem 10.1]{Per02} is also a corollary of Theorem \ref{prop: almost mono} and Theorem \ref{thm: Nash depends on nu}.

\begin{Corollary}[{\cite[Theorem 1.8]{CMZ21a}}]\label{CMZ21A1.8}
Assume that $[-2r^2,0]\subseteq I$. Let $x_0$ be a point on $M$ and $(z,-r^2)$ be an $H_n$-center of $(x_0,0)$. Furthermore, assume that
$$|B_{-r^2}(z,r)|_{-r^2}\geq \alpha r^n,$$
where $\alpha$ is a positive constant. Then we have 
$$\N_{x_0,0}(r^2)\geq -\beta(n,\alpha),$$
where $\beta$ is a positive constant depending only on $n$ and $\alpha$.
\end{Corollary}

\begin{proof}[{Proof of Corollary \ref{CMZ21A1.8} assuming Theorem \ref{thm: Nash depends on nu}} and Corollary \ref{Corollary2}]
By applying the maximum principle to the evolution equation of $R$, we have
$$r^2 R_{g_{-r^2}}\geq -\frac{n}{2}.$$
By Corollary \ref{Corollary2}, we have 
$$\mu\left(B_{-r^2}\left(z,8\sqrt{H_n}r\right),g_{-r^2},r^2\right)\geq -C(n,\alpha).$$
The conclusion then follows from Theorem \ref{thm: Nash depends on nu}.
\end{proof}

\subsection{Perelman's noncollapsing improving and pseudolocality theorems}

As we have mentioned above, previous local monotonicity theorems can be understood as special examples of ours, in which the local geometric assumptions control the position of an $H_n$-center. Perelman's noncollapsing improving theorem \cite[Theorem 8.2]{Per02} is one of them. Recall that, by a second-variation argument, Perelman discovered that the distance function along Ricci flow is a supersolution to a linear parabolic equation, and subsequently proved that the geometric information in a space-time neighborhood of a point determines the geometric information of a larger space-time neighborhood. Bamler \cite[Theorem 3.5]{Bam20a} has refined the second variation argument, and the consequence is a much better property called $H_n$-concentration which we already mentioned above. As we have seen in Subsection 1.2, it is along the $P^*$-parabolic neighborhoods, not the classical parabolic neighborhoods, that the geometric information is diffused, and the local geometric assumptions near a point is nothing but that which controls the shape of a $P^*$-parabolic neighborhood.

Therefore, it is no surprising that Perelman's noncollapsing improving theorem is a consequence of our local monotonicity formula. In fact, the point emphasized in the above paragraph is inspired by the work of Jian \cite{J21}, 
who provided an improved version of \cite[Theorem 8.2]{Per02} and \cite[Theorem 1.1]{W18}. 
Hereby we also prove a stronger version of Jian's result, which is simply a combination of Jian \cite{J21} and Corollary \ref{Corollary3} above.

\begin{Theorem}[Noncollapsing improving]\label{Theorem2}
Assume that $[-2r^2,0]\subseteq I$. Let $x_0\in M$ be a fixed point. If
\begin{align*}
    \int_{-r^2}^0 \sqrt{|t|}R(x_0,t)\, dt 
    & \le Ar,\\
    |B_{-r^2}(x_0,r))|_{-r^2}& \ge A^{-1}r^n,
\end{align*}
then
\[
\nu(B_0(x_0,A r),g_0, A^2r^2)\ge -C(n,A).
\]
As a consequence, there is a constant $\kappa=\kappa(n,A)>0$, such that for any ball $B:=B_0(y,\rho)$ satisfying $y\in B_0(x_0,Ar)$ and $\rho\in(0,Ar]$, we have
\[
    \sup_B R_{g_0}\le \rho^{-2} \implies |B|_{0} \ge \kappa \rho^n.
\]
\end{Theorem}

Another application of our local monotonicity formula is to simplify the proof of Perelman's pseudolocality theorem \cite[Theorem 10.1]{Per02}.  By Theorem \ref{thm: Nash depends on nu}, if the geometry at an $H_n$-center is regular enough, namely, almost Euclidean, then the Nash entropy is almost equal to zero. Furthermore, Bamler's $\varepsilon$-regularity theorem \cite[Theorem 10.3]{Bam20a} shows that the smallness of Nash entropy implies the regularity at the base point. Combining these two facts, Perelman's pseudolocality theorem follows naturally. Here we present B.~Wang's improved version of Perelman's pseudolocality theorem \cite[Theorem 1.2]{W20}.

\begin{Theorem}[Pseudolocality theorem {\cite{Per02,W20}}]\label{BWang2}
For any $\alpha\in(0,\frac{1}{100n})$, there is a $\delta=\delta(n,\alpha)>0$, such that the following holds. Assume that $[0,T]\subset I$. Let $x_0\in M$ be a point satisfying 
\begin{align}\label{pseudo1}
    \inf_{0<t\leq T}\mu\left(B_0\left(x_0,\delta^{-1}\sqrt{t}\right),g_0,t\right)\geq -\delta^2.
\end{align}
Then, for any $t\in(0,T]$ and any $x\in B_t\left(x_0,\alpha^{-1}\sqrt{t}\right)$, we have
\begin{align}
    |{\Rm_{g_t}}|(x)&\leq \alpha t^{-1}\label{pseudo2}
    \\
    \inf\left\{\rho^{-n}|B_t(x,\rho)|_t:\rho\in\left(0,\alpha^{-1}\sqrt{t}\right)\right\}&\geq (1-\alpha)\omega_n,\label{pseudo3}
    \\
    \operatorname{inj}_{g_t}(x)&\geq \alpha^{-1}\sqrt{t},\label{pseudo4}
\end{align}
where $\omega_n$ is the volume of the $n$-dimensional unit disk in the Euclidean space.
\end{Theorem}

\emph{Remark.} Another pseudolocality theorem of Perelman \cite[Theorem 10.3]{Per02} (c.f. \cite{Lu10}) can also be simplified by our method. One may combine the method in the proof of Theorem \ref{BWang2} with the argument in \cite{Lu10} to obtain a proof slightly simpler than Lu's---at least, like the proof of Theorem \ref{BWang2}, the point picking argument can be avoided, and the division into multiple cases is not necessary. The details are left to the readers.

\subsection{Application to ancient Ricci flows}

Finally, we show that our uniform Sobolev inequality for ancient Ricci flows \cite{CMZ21b} is merely a consequence of our local $\nu$-functional estimate in Theorem \ref{thm: nu ge nash}.

\begin{Theorem}\label{AncientSobolev}
(\cite[Theorem 1.1]{CMZ21b}) Let $I=(-\infty,0]$ and assume that $(M^n,g_t)_{t\in I}$ has uniformly bounded Nash entropy, namely, there is a point $(x_0,t_0)\in M\times (-\infty,0]$ such that
$$\mu_\infty:=\lim_{\tau\to\infty}\N_{x_0,t_0}(\tau)>-\infty.$$
Then we have
$$\inf_{t\leq 0}\nu(g_t)=\mu_\infty>-\infty.$$
\end{Theorem}

The following corollary is then a standard consequence by applying the computations in \cite{Zhq10,Ye15}. 
\begin{Corollary}(\cite[Corollary 1.2]{CMZ21b}) 
Under the same condition as the above theorem, we have
\begin{enumerate}[(1)]
    \item Logarithmic Sobolev inequality: for any compactly supported locally Lipschitz function $u$ on $(M,g_t)$, where $t\leq 0$, and positive scale $\tau>0$, we have
    \begin{align*}
    \int_M u^2\log u^2dg_t-\log\left(\int_Mu^2dg_t\right)\int_M u^2dg_t+\left(\mu_\infty+n+\frac{n}{2}\log(4\pi\tau)\right)\int_Mu^2dg_t&
    \\
    \leq \tau\int_M(4|\nabla u|^2+Ru^2)dg_t&.
    \end{align*}
    \item Sobolev inequality: for any compactly supported locally Lipschitz function $u$ on $(M,g_t)$, where $t\leq 0$, we have
    \begin{equation*}
    \left(\int_M |u|^{\frac{2n}{n-2}}dg_t\right)^{\frac{n-2}{n}}\leq C(n)e^{-\frac{2\mu_\infty}{n}}\int_M(4|\nabla u|^2+Ru^2)dg_t.
    \end{equation*}
\end{enumerate}
\end{Corollary}

Recall that in \cite{CMZ21b}, we have applied Bamler's tangent flow at infinity as an ``initial manifold'' to apply Q.~Zhang's argument \cite{Zhq10}. However, this requires that Bamler's theory of singular spaces \cite{Bam20c} can be extended to noncompact manifolds. This, though almost certainly true, is not yet fully technically verified 
in
every detail. Indeed, Bamler \cite{Bam21} has sketched how his theory can be generalized to noncompact manifolds. The proof of Theorem \ref{AncientSobolev} which we shall present in this article does not rely on tangent flow at infinity, and hence is much more technically transparent and reader-friendly, for it is nothing but taking $r$ to infinity in formula (\ref{localnu}).

To help the reader navigate our paper, the titles of Sections 3---7 are made identical to the titles of the subsections of Section 1. For a result presented in a subsection of Section 1, one may expect its proof to appear in the section bearing the same title.

\bigskip

\textbf{Acknowledgement.} The third-named author would like to thank Professor Jiaping Wang for many helpful discussions, and for suggesting this problem. The authors would like to thank Professor Richard Bamler for suggesting a statement as presented in Theorem \ref{thm: Nash depends on nu} after the first draft of this paper. The authors would like to thank Professor Richard Bamler and Professor Bennett Chow for their constant and helpful communicating of ideas.

\section{Preliminaries}

\subsection{Perelman's entropy and Nash entropy}
 Let $(M^n,g)$ be a Riemannian manifold, $\tau>0$, and $u$ be a 
 nonnegative function with unit integral 
 such that $\sqrt{u}$ is smooth, 
 then we define the $\bW$ functional as 
\[
    \bW(g,u,\tau)
    := \int_M \left(\tau\left( |\nabla\log u|^2 + R\right) - \log u\right)u\, dg - \frac{n}{2}\log(4\pi\tau)-n.
\]
In fact, writing $u:=(4\pi\tau)^{-\frac{n}{2}}e^{-f}$, it is easy to observe that $\bW$ is the same as Perelman's $\W$-functional
\[
\W(g,f,\tau):= \int_M\left(\tau\left(|\nabla f|^2+R\right)+f-n\right)(4\pi\tau)^{-\frac{n}{2}}e^{-f}dg.
\]
For any region $\Omega\subseteq M,$ following \cite{W18}, we define
\begin{align*}
     \mu(\Omega,g,\tau)
    &:= \inf \left\{\bW(g,u,\tau):
    u\ge 0,\sqrt{u}\in C_0^\infty(\Omega),
    \int_M u\, dg=1
    \right\},\\
    \nu(\Omega,g,\tau)
    &:= \inf_{0<s\le\tau} \mu(\Omega,g,s).
\end{align*}
Indeed, $\nu(\Omega,g,\tau)$ is a local Sobolev constant of the region $\Omega$. If $\Omega= M$ and $\tau=\infty$, then we shall denote by $\nu(g):=\nu(M,g,\infty)$ the $\nu$-functional of Perelman \cite{Per02}.

 By Perelman's monotonicity formula, if $u:=(4\pi\tau_t)^{-\frac{n}{2}}e^{-f}$ is a positive solution to the conjugate heat equation $$\Box^*_tu:=-\partial_tu-\Delta_{g_t} u+R_{g_t}u=0$$ with unit integral, where $\frac{d}{dt}\tau_t=-1$, then $\bW(g_t,u(\cdot,t),\tau_t)$ is, in general, non-decreasing in $t$. More precisely, assuming the validity of integration by parts at infinity, we have
\begin{equation}\label{W monotone}
    \frac{d}{dt}\bW(g_t,u(\cdot,t),\tau_t)=\int_M 2\tau_t\left|\Ric_{g_t}+\nabla^2f-\tfrac{1}{2\tau_t}g_t\right|^2u_tdg_t\geq 0.
\end{equation}
Here the conjugate heat operator $\Box^*_t$ is the conjugate of the heat operator
$$\Box_t:=\partial_t-\Delta_{g_t},$$
since, provided that integration by parts is valid, we have
$$\frac{d}{dt}\int_M uv\,dg_t=\int_M(v\Box_tu-u\Box^*_tv)\,dg_t.$$

Next, we introduce the definition of Perelman's entropy and the Nash entropy. Let $(x_0,t_0)\in M\times I$ be a fixed point in space-time. Then, letting $u_t=(4\pi(t_0-t))^{-\frac{n}{2}}e^{-f_t}:=K(x_0,t_0\,|\,\cdot,t)$ be the conjugate heat kernel based at $(x_0,t_0)$, Perelman's entropy and the Nash entropy based at $(x_0,t_0)$ are respectively defined as
\begin{align}
    \W_{x_0,t_0}(\tau)&=\bW(g_{t_0-\tau},u_{t_0-\tau},\tau),
    \\
    \N_{x_0,t_0}(\tau)&=\int_M f_{t_0-\tau}\,d\nu_{x_0,t_0\,|\,t_0-\tau}-\frac{n}{2},\label{def_nash}
\end{align}
where $\tau>0$ and $t_0-\tau\in I$, and the evolving probability measure 
$$\nu_{x_0,t_0\,|\,t}(\Omega):=\int_\Omega u_t\,dg_t,\quad \Omega\subseteq M\text{ is measurable, and } t\in(-\infty,t_0]\cap I,$$
is also called the \emph{conjugate heat kernel} based at $(x_0,t_0)$ when there is no ambiguity. It is well-known that both $\W_{x_0,t_0}(\tau)$ and $\N_{x_0,t_0}(\tau)$ are monotonically decreasing in $\tau$, and that
\begin{align}\label{Nash greater than Pen}
    0\geq \N_{x_0,t_0}(\tau)\geq \W_{x_0,t_0}(\tau)\qquad\text{ for all } \tau>0 \quad\text{ and } \quad t_0-\tau\in I.
\end{align}

\subsection{Bamler's Harnack inequality for the Nash entropy}

Bamler's Harnack inequality \cite[Corollary 5.11]{Bam20a} enables us to compare the Nash entropy based not only at different points on the same time-slice, but also at points on different time-slices.  Let us recall this result. Note that the second and the third authors \cite[Corollary 4.5]{MZ21} have already generalized this theorem to the noncompact setting, assuming only bounded curvature on each compact time interval.

\begin{Theorem}(\cite[Corollary 5.11]{Bam20a})\label{Harnack}
If $R_{g_{t^*}}\geq R_{\operatorname{min}}$, $s<t^*\leq \min\{t_1,t_2\}$, and $s$, $t^*$, $t_1$, $t_2\in I$, then for any $x_1$, $x_2\in M$, we have
\begin{equation*}
    \N_{x_1,t_1}(t_1-s)-\N_{x_2,t_2}(t_2-s)\leq\left(\frac{n}{2(t^*-s)}-R_{\operatorname{min}}\right)^{\frac{1}{2}}\dist^{g_{t^*}}_{W_1}(\nu_{x_1,t_1\,|\,t^*},\nu_{x_2,t_2\,|\,t^*})+\frac{n}{2}\log\left(\frac{t_2-s}{t^*-s}\right).
\end{equation*}
\end{Theorem}

In the statement of the above theorem, $\dist_{W_1}$ is the $W_1$-Wasserstein distance between two probability measures on a metric space. Recall that the $P^*$-parabolic neighborhood of Bamler \cite[Definition 9.2]{Bam20a} is defined in terms of the $W_1$-Wasserstein distance. Let $(x_0,t_0)\in M\times I$. Then, for all $A>0$ and $T^-$, $T^+\geq0$,  $P^*(x_0,t_0\,|\,A,-T^-,T^+)$ is the set of all points $(x,t)\in M\times (I\cap [t_0-T^-,t_0+T^+])$, satisfying
$$\dist_{W_1}^{g_{t_0-T^-}}(\nu_{x_0,t_0\,|\, t_0-T^-},\nu_{x,t\,|\,t_0-T^-})<A.$$
Therefore, Theorem \ref{Harnack} can be stated in the way that the difference of the Nash entropy is determined by the ``$P^*$-distance'' between the base points.

The notion of $H_n$-center, where $H_n:=\frac{(n-1)\pi^2}{2}+4$, is exactly that which characterizes ``shortest $P^*$-distance'' (\cite[Definition 3.10]{Bam20a}). Let us recall that $(z,s)\in M\times I$ is an $H_n$-center of $(x,t)\in M\times I$, where $s<t$, if
$$\operatorname{Var}(\delta_z,\nu_{x,t\,|\,s})\leq H_n(t-s),$$
where $\operatorname{Var}$ is the variance of two probability measures (\cite[Definition 3.1]{Bam20a}). By the definitions of variance and $W_1$-Wasserstein distance, we also have
\begin{equation}\label{HNW1}
    \dist_{W_1}^{g_s}(\delta_z,\nu_{x,t\,|\,s})\leq\sqrt{H_n(t-s)}.
\end{equation}
Therefore, a path of $H_n$-centers, heuristically speaking, can be understood as a ``$P^*$-geodesic''. Given an arbitrary point $(x,t)\in M\times I$ and an arbitrary time $s\in I\cap(-\infty,t)$, there always exists a point $z$, such that $(z,s)$ is an $H_n$-center of $(x,t)$.

\subsection{Reduced distance as a ``sub''-conjugate heat kernel}

Perelman's reduced distance, though a very interesting and important object of study in its own right, is also a central tool in the analysis of the conjugate heat equation. Let $(x_0,t_0)\in M\times I$ be a fixed point in space-time, then the reduced distance function based at $(x_0,t_0)$ is defined to be
$$\ell_{x_0,t_0}(x,\tau):=\frac{1}{2\sqrt{\tau}}\inf_\gamma\int_0^\tau\sqrt{\eta}\left(R_{g_{t_0-\eta}}(\gamma(\eta))+|\gamma'(\eta)|^2_{g_{t_0-\eta}}\right)d\eta,$$
where $x\in M$, $\tau>0$, $t_0-\tau\in I$, and the infimum is taken over all piecewise smooth curves $\gamma:[0,\tau]\to M$ satisfying $\gamma(0)=x_0$ and $\gamma(\tau)=x$.

The following Theorem is a restatement of \cite[Corollary 9.5]{Per02}.

\begin{Theorem}\label{subsolution}
For any $(x,t)$ and $(x_0,t_0)\in M\times I$ with $t<t_0$, we have
\begin{align*}
    (4\pi(t_0-t))^{-\frac{n}{2}}e^{\ell_{x_0,t_0}(x,t)}\le K(x_0,t_0\,|\,x,t),
\end{align*}
where $K$ is the fundamental solution to the conjugate heat equation.
\end{Theorem}

\section{Local Sobolev inequalities on Ricci flow}

In this section, we present the proof of Theorem \ref{thm: nu ge nash}, Corollary \ref{Qzhang}, and Corollary \ref{Corollary2}, where the last two results follow immediately from the first. Our main technique is Bamler's Harnack inequality of the Nash entropy \cite[Corollary 5.11]{Bam20a} (see Theorem \ref{Harnack} above), which shows how the Nash entropy is dependent on the base point. Given a bound of the Nash entropy based at a fixed point $(x_0,t_0)$ in space-time, we may, by applying \cite[Corollary 5.11]{Bam20a}, show that the Nash entropy based everywhere nearby is similarly bounded. Since the Nash entropy can also be regarded as a concave functional of the conjugate heat kernel, we may show, by applying Jensen's inequality, that any test function whose support is in a disk center at $(x_0,t_0)$ has bounded ``Nash functional'', which also implies a bound of the local $\nu$-functional up to a certain scale.

 By parabolic rescaling, we assume that $r=1$, and we let $\tau>0,A<\infty$ be arbitrarily fixed constants. We shall pick an arbitrary test function and  verify that its $\bW$ functional is bounded from below by the Nash entropy in the way as stated in the theorem. To this end, let $u_0$ be an arbitrary nonnegative test function such that
$\sqrt{u_0}\in C_0^\infty(B_0(x_0,A))$ and $\int_M u_0 \, dg_0=1.$ 
We solve the backward conjugate heat equation
\[
    \Box_t^* u_t:=(- \partial_t - \Delta_{g_t} + R_{g_t}) u_t = 0\quad \text{ on } \quad M\times [-1,0],
\]
with the initial data being $u_0$ at time $t=0.$ Then for all $t\in[-1,0]$, $u_t$ can be written as
\begin{equation}\label{the CHF}
    u_t(x)=\int_M K(y,0\,|\,x,t)u_0(y)dg_0(y),
\end{equation}
where $K$ is the fundamental solution to the conjugate heat equation. Then, the evolving probability measure 
\[
    \mu_t(\Omega):=\int_{\Omega}u_tdg_t,\quad \Omega\subseteq M\text{ is measurable and }t\in[-1,0]
\]
is what Bamler \cite{Bam20b} calls a \emph{conjugate heat flow}.
For $t\in [-1,0)$, we define 
\begin{align*}
    \tau_t:= \tau & - t,\quad
    u_t =: (4\pi \tau_t)^{-n/2} e^{-f_t},\\
    \bN(t) &:= \int_M f_t \, d\mu_t - \frac{n}{2},\\
    \bW(t) &:= \int_M \left(\tau_t(|\nabla f_t|^2+R_{g_t}) + f_t - n\right)d\mu_t.
\end{align*}

\begin{Lemma}
\label{lem: W conv}
\begin{align}
   \bW(0):=\lim_{t\to 0^-} \bW(t) &= \bW(g_0,u_0,\tau), \label{eq: lim of W}
   \\
   -\frac{d}{dt}\left(\tau_t \bN(t)\right)&= \bW(t),\label{eq: der of Nash}
   \\
   \frac{d}{dt} \bW(t) &\ge 0.\label{eq: der of W}
\end{align}
\end{Lemma}
\begin{proof}
The proof of the lemma, especially of (\ref{eq: lim of W}), is not essentially different from \cite[\S 4]{W18}, and a detailed treatment can be found in \cite[\S 9]{CMZ21a}. Indeed, by \cite[Theorem 26.25 and Theorem 26.31]{RFV3}, there is a constant $C$ depending only on the local geometry bound in $B_0(x_0,2A)\times [-1,0]$, such that
\begin{eqnarray}\label{gaussian}
\frac{1}{C(t-s)^{\frac{n}{2}}}\exp\left(-\frac{C\dist_t^2(x,y)}{(t-s)}\right)\leq K(x,t\,|\,y,s)\leq \frac{C}{(t-s)^{\frac{n}{2}}}\exp\left(-\frac{\dist_t^2(x,y)}{C(t-s)}\right),
\end{eqnarray}
whenever $-1\leq s<t\leq 0$ and either $x$ or $y$ is in $B_0(x_0,A)$. As a consequence, we have \begin{align}\label{quadratic_nonsense}
u(x,t)\leq \frac{C}{(-t)^{\frac{n}{2}}}\exp\left(-\frac{\dist^2_{0}(x,B_0(x_0,A))}{C(-t)}\right),\quad\text{for all $(x,t)\in M\times[-1,0)$.}
\end{align}
One may then follow the argument in \cite[\S 9]{CMZ21a} line by line to obtain (\ref{eq: lim of W}).

For (\ref{eq: der of W}), one needs only to observe that the bounds in (\ref{quadratic_nonsense}) together with some standard gradient estimates (c.f. \cite[Theorem 10]{EKNT08} or \cite[Lemma 9.11]{CMZ21a}) enable integration by parts at infinity. Hence, (\ref{eq: der of W}) is a consequence of Perelman's monotonicity formula (\ref{W monotone}).

Finally, given the validity of integration by parts at infinity, (\ref{eq: der of Nash}) follows from standard computation; one may refer to \cite{CMZ21b} for more details.
\end{proof}

\begin{proof}[\textbf{Proof of Theorem \ref{thm: nu ge nash}}]
Our goal is to estimate the lower bound of $\bW(0)$. 
By \eqref{eq: lim of W}---\eqref{eq: der of W}, we compute
\begin{align*}
    &\  
        \tau_{-1} \bN(-1) - \tau_0 \bN(0)\\
    =&\  -\int_{-1}^0 \frac{d}{dt}\left(\tau_t \bN(t)\right)\, dt
    =  \int_{-1}^0 \bW(t)\, dt\\
    \le &\ \bW(0) = \bW(g_0,u_0,\tau),
\end{align*}
where we have defined 
$$\bN(0) := \lim_{t\to 0^-}\bN(t).$$
Therefore, it remains to estimate $\bN(-1)$ and $\bN(0)$.

First, we estimate $\bN(0)$. By the maximum principle, for $s\in (-1,0],$ we have
\begin{equation}\label{Rlowerbound}
        R_{g_s} \ge -\frac{n}{2(s+1)}\quad \text{ on }\quad  M.
\end{equation}
For each $s\in (-1,0)$ close to $0,$ we have
\begin{align*}
    \bN(s) &= \bW(s) + \frac{n}{2} - \int \tau_s\left( |\nabla f_s|^2 + R_{g_s}\right)\, d\mu_s\\
    & \le \bW(s)  + \frac{n}{2} + \frac{n\tau_s}{2 (s+1)}.
\end{align*}
By taking $s\to 0^-,$ we have
\[
    \bN(0) \le \bW(0) + \frac{n}{2}(1+\tau).
\]

Then
\begin{align*}
    \bW(0) & \ge 
        \tau_{-1} \bN(-1) - \tau_0 \N(0)\\
    & \ge 
 (1+\tau)\bN(-1) - \tau \bW(0) - \frac{n}{2}\tau(1+\tau),
\end{align*}
and hence 
\begin{equation}\label{something1}
    \bW(0)\ge \bN(-1) - \frac{n}{2}\tau.
\end{equation}
It remains to estimate $\bN(-1).$ To this end, we recall the definition of $\bN$:
\begin{align}\label{bound of N-1}
    \bN(-1)
    &= \int f_{-1} u_{-1} \, dg_{-1} - \frac{n}{2}\\\nonumber
    &= - \int u_{-1} \log u_{-1}\, dg_{-1}
    - \frac{n}{2} - \frac{n}{2} \log(4\pi (1+\tau)).
\end{align}
By (\ref{the CHF}), we have
\[
    u_{-1}(x) = \int K(y,0\,|\,x,-1)u_0(y)\, dg_0(y)=\int K(y,0\,|\,x,-1)\, d\mu_0(y),
\]
where $\mu_0$ is a probability measure. 
Hence, by Jensen's inequality, we have
\[
    u_{-1}\log u_{-1}(x)
    \le \int K(y,0\,|\,x,-1)
    \log K(y,0\,|\,x,-1) \, d\mu_0(y).
\]
It follows from integrating both sides with the measure $dg_{-1}$ that
\begin{align}\label{nonsense}
    &\ \int u_{-1}\log u_{-1}(x)\, dg_{-1}(x)\\\nonumber
    \le&\  \int \int K(y,0\,|\,x,-1)
    \log K(y,0\,|\,x,-1) \, d\mu_0(y) dg_{-1}(x)\\\nonumber
    =&\  \int \int K(y,0\,|\,x,-1)
    \log K(y,0\,|\,x,-1) \,  dg_{-1}(x)d\mu_0(y)\\\nonumber
    =&\  -\int \N_{y,0}(1) \, d\mu_0(y)- \frac{n}{2} - \frac{n}{2}\log(4\pi).
\end{align}
Here, it is easy to verify that 
the change of order of the integration is valid, since $\mu_0$ is supported in $B_0(x_0,A)$ and since, by (\ref{gaussian}), $K(y,0\,|\,\cdot,-1)$ has rapid decay at infinity.


Combining (\ref{bound of N-1}) and (\ref{nonsense}), we have
\begin{equation}\label{something2}
     \bN(-1)
    \ge \int \N_{y,0}(1) \, d\mu_0(y) - \frac{n}{2} \log (1+\tau).
\end{equation}
We shall estimate the integral on the right-hand-side. Let us fix an arbitrary $y\in {\rm spt}(\mu_0)\subseteq B_0(x_0,A)$. By the definition of $W_1$-distance, we have
\[
    \dist_{W_1}^{g_{0}}(\nu_{y,0\,|\,0}, \nu_{x_0,0\,|\,0})
    =\dist_{W_1}^{g_0}(\delta_y,\delta_{x_0})= \dist_{g_0}(x_0,y) < A.
\]
We may then apply Theorem \ref{Harnack} with $x_0\to x_1$, $y\to x_2$, $0\to t_1=t_2=t^*$, $-1\to s$, and obtain
\begin{align}\label{something3}
    \N_{x_0,0}(1)
    &\le \N_{y,0}(1) + \left(\tfrac{n}{2} - \inf R_{g_0}\right)^{\frac{1}{2}} \dist_{W_1}^{g_0}(\nu_{y,0\,|\,0}, \nu_{x_0,0\,|\,0})
    + \frac{n}{2}\log \frac{1}{1}\\\nonumber
    &\le  \N_{y,0}(1) + \sqrt{n}A,
\end{align}
where we have also applied (\ref{Rlowerbound}).
It follows from (\ref{something1}), (\ref{something2}), and (\ref{something3}) that
\begin{align*}
    \bW(0) &\ge \bN(-1) - \frac{n}{2}\tau\\
    &\ge \int \N_{y,0}(1)\, d\mu_0(y)  - \frac{n}{2}\tau - \frac{n}{2}\log(1+\tau)\\
    &\ge \N_{x_0,0}(1) - \sqrt{n}A - \frac{n}{2}\tau - \frac{n}{2}\log(1+\tau).
\end{align*}
Since the test function $u_0$ is arbitrarily fixed, we have
\begin{equation*}
    \mu\left(B_0(x_0,A), g_0, \tau \right)\ge \N_{x_0,0}(1) - \sqrt{n}A - \frac{n}{2}\tau - \frac{n}{2}\log(1+\tau).
\end{equation*}
To see that the local $\nu$-functional is also bounded, one needs only to observe that, for any $s\in(0,\tau]$, we have
\begin{align*}
     \mu\left(B_0(x_0,A), g_0, s \right)&\ge \N_{x_0,0}(1) - \sqrt{n}A - \frac{n}{2}s - \frac{n}{2}\log(1+s)\\
     &\geq \N_{x_0,0}(1) - \sqrt{n}A - \frac{n}{2}\tau - \frac{n}{2}\log(1+\tau).
\end{align*}
This finishes the proof of the theorem. 
\end{proof}

\begin{proof}[\textbf{Proof of Corollary \ref{Qzhang}}]
We still let $r=1$. The idea of the proof is to use the local geometric conditions (\ref{localgeometrybounds}) to estimate the distance between $(x_0,1/2)$ and an $H_n$-center of $(x_0,1)$. First of all, applying (\ref{localgeometrybounds}) and the standard distortion estimates of distance and volume, we may easily obtain
\begin{gather}\label{localgeometrybounds1}
    |B_{1/2}(x_0,c_n)|_{1/2}\ge c(n,A),
    \\\nonumber
    |\Ric|\le \frac{1}{n^2}\quad\text{ on }\quad B_{1/2}(x_0,c_n)\times\left[\tfrac{1}{2},1\right],
\end{gather}
where $c_n$ is a positive constant depending only on $n$. We shift the base time from $0$ to $1/2$ for the purpose of applying \cite[Theorem 1.8]{CMZ21a}.

Let us fix any point $x\in B_{1/2}(x_0,c_n)$, and let $\gamma:[0,1/2]\to M$ be a minimizing $g_{1/2}$-geodesic with constant speed connecting $x_0$ and $x$. Then, by (\ref{localgeometrybounds1}), we estimate
$$|\gamma'(s)|^2_{g_{1-s}}\leq e^{\frac{1}{n^2}}|\gamma'(s)|^2_{g_{1/2}}\leq 4c^2_ne^{\frac{1}{n^2}}\quad\text{ for all }\quad s\in\left[0,\tfrac{1}{2}\right],$$
and consequently
$$\ell_{x_0,1}\left(x,\tfrac{1}{2}\right)\le\frac{1}{2}\int_0^{1/2}\sqrt{\eta}\left(R_{g_{1-\eta}}(\gamma(\eta))+|\gamma'(s)|^2_{g_{1-\eta}}\right)d\eta\leq C(n).$$
Hence, Theorem \ref{subsolution} implies that
$$K\left(x_0,1\,\big|\,x,\tfrac{1}{2}\right)\geq (4\pi)^{-\frac{n}{2}}e^{-\ell_{x_0,1}(x,1/2)}\ge c(n)\quad\text{ for all }\quad x\in B_{1/2}(x_0,c_n),$$
and by (\ref{localgeometrybounds1}), we have
$$\nu_{x_0,1\,|\,1/2}\left(B_{1/2}(x_0,c_n)\right)\geq c(n)|B_{1/2}(x_0,c_n)|_{1/2}\geq c_0(n,A).$$

Letting $(z,1/2)$ be an $H_n$-center of $(x_0,1)$, we have, by \cite[Proposition 3.13]{Bam20a},
$$\dist_{1/2}(x_0,z)\leq c_n+\sqrt{\tfrac{1}{2c_0(n,A)}H_n}\leq C(n,A).$$
By (\ref{localgeometrybounds1}) again, we have
$$\left|B_{1/2}\left(z,C(n,A)+c_n\right)\right|_{1/2}\geq \left|B_{1/2}\left(x_0,c_n\right)\right|\ge c(n,A).$$
It then follows from \cite[Theorem 1.8]{CMZ21a} and the remark thereof that
$$\N_{x_0,1}(1/2)\geq -C(n,A).$$
The corollary follows immediately from Theorem \ref{thm: nu ge nash}.

\end{proof}

\begin{proof}[\textbf{Proof of Corollary \ref{Corollary2}}]
We still assume $r=1$ by applying a parabolic scaling.
We shall only prove the bound of the local $\nu$-functional, and the noncollapsing result is a straightforward consequence.
By (\ref{Rlowerbound}) again, $R_{g_s}\ge - n$ for $s\in [-1/2,0].$ Hence, \cite[Theorem 8.1]{Bam20a} implies that
\begin{equation}\label{X}
     \alpha \le |B_{0}(x_0,1)|_0 
    \le C(n) \exp(\N_{x_0,0}(1/2)).
\end{equation}
Applying Theorem \ref{thm: nu ge nash}, we have
\begin{align*}
    \nu(B_{0}(x_0,A),g_0,\tau)
    &\ge \N_{x_0,0}(1/2)
    - 2\sqrt{n}A - 4n\tau\\
    &\ge \log \alpha - 4n(A +\tau) - C(n).
\end{align*}
   This finishes the proof. 
\end{proof}


\section{The ``reaction-diffusion'' property of the Ricci flow}

In this section, we prove Corollary \ref{Corollary3} and Theorem \ref{HeatHarnack}. The proof of Corollary \ref{Corollary3} is no more than an application of Theorem \ref{thm: nu ge nash} and Theorem \ref{Harnack}. In fact, given the assumption in Corollary \ref{Corollary3}, Bamler's Harnack inequality shows that the Nash entropy based at any point in $P^*(x_0,0\,|\,Ar,0,A^2r^2)$ is bounded, and this corollary follows from Theorem \ref{thm: nu ge nash}. The proof of Theorem \ref{HeatHarnack} uses the fact that the conjugate heat kernel concentrates around an $H_n$-center, and that a solution to the heat equation can be viewed as a convolution with a conjugate heat kernel.

\begin{proof}[\textbf{Proof of Corollary \ref{Corollary3}}]
The proof is no more than an application of the definition of $P^*$-parabolic neighborhood. As usual, we assume, by parabolic rescaling, that $r=1$. Let $(y,s)$ be an arbitrary point in $P^*(x_0,0\,|\,A,0,A^2)$, and we shall estimate the Nash entropy based at $(y,s)$. By the definition of $P^*$-parabolic neighborhood, we have
$$\dist_{W_1}^{g_0}(\delta_{x_0},\nu_{y,s\,|\,0})\leq A.$$
Since (\ref{Rlowerbound}) implies 
$$R_{g_0}\geq -\frac{n}{2}\quad \text{ on }\quad M,$$
we may apply Theorem \ref{Harnack} with $(x_0,0)\to (x_1,t_1)$, $(y,s)\to (x_2,t_2)$, $0\to t^*$, $-1\to s$, $-\frac{n}{2}\to R_{\operatorname{min}}$, and this yields
\begin{align*}
    \N_{x_0,0}(1)&\leq \N_{y,s}(1+s)+\left(\frac{n}{2(0-(-1))}-(-\frac{n}{2})\right)^{\frac{1}{2}}\dist^{g_0}_{W_1}(\nu_{y,s\,|\,0},\delta_{x_0})+\frac{n}{2}\log\left(\frac{s-(-1)}{0-(-1)}\right)
    \\
    &\leq \N_{y,s}(1+s)+\sqrt{n}A+\frac{n}{2}\log(1+A^2),
\end{align*}
where we have used the fact that $s\in[0,A^2]$. Therefore, we have
\begin{align*}
    \N_{y,s}(1)&\geq \N_{y,s}(1+s)\geq \N_{x_0,0}(1)-\sqrt{n}A-\frac{n}{2}\log(1+A^2),
\end{align*}
and the conclusion follows from Theorem \ref{thm: nu ge nash}.
\end{proof}

\begin{proof}[\textbf{Proof of Theorem \ref{HeatHarnack}}]
By parabolic scaling, we assume $r=1$. Let us define the function $\tilde H: M\times [-1,0]\to\mathbb R$ as
$$\tilde H(x,t):=\int_{\bar B_{-1}(z,A)}H(\cdot,-1)\,d\nu_{x,t\,|\,-1}\quad\text{ for }\quad (x,t)\in M\times[-1,0].$$
Obviously, since $\tilde H$ solves the heat equation, and since $H(\cdot,-1)\ge\tilde H(\cdot,-1)$ everywhere on $M$, we have that $H\geq \tilde H$ everywhere on $M\times(0,1]$ by the maximum principle. Hence, we need only to estimate $\tilde H$ on $\bar B_0(x_0,\Lambda)$.

By \cite[Proposition 3.14]{Bam20a}, we have
\begin{align*}
    \nu_{x_0,0\,|\,-1}(\bar B_{-1}(z,A))\geq 1-2\exp\left(-\tfrac{1}{8}\left(A-\sqrt{2H_n}\right)^2\right).
\end{align*}
Therefore, regarding 
$$(x,t)\mapsto\nu_{x,t\,|\,-1}(\bar B_{-1}(z,A))$$
as a solution to the heat equation, and applying \cite[Theorem 4.1]{Bam20a}, we have
\begin{align}\label{ultimatenonsense1}
    \nu_{x,0\,|\,-1}(\bar B_{-1}(z,A))&\geq \Phi\left(\Phi^{-1}\left(\nu_{x_0,0\,|\,-1}(\bar B_{-1}(z,A))\right)-\Lambda\right)
    \\\nonumber
    &\geq \Phi\left(\Phi^{-1}\left(1-2\exp\left(-\tfrac{1}{8}\left(A-\sqrt{2H_n}\right)^2\right)\right)-\Lambda\right),
\end{align}
for all $x\in B_0(x_0,\Lambda)$. We will now estimate the right-hand-side. Let
$$\eta=\Phi^{-1}\left(1-2\exp\left(-\tfrac{1}{8}\left(A-\sqrt{2H_n}\right)^2\right)\right).$$
Then, by the definition of $\Phi$, we have
\begin{align*}
    2\exp\left(-\tfrac{1}{8}\left(A-\sqrt{2H_n}\right)^2\right)&=\int_\eta^\infty(4\pi)^{-\frac{1}{2}}e^{-\frac{t^2}{4}}dt
    \\
    &\ge \int_\eta^{\eta+4\sqrt{\pi}}(4\pi)^{-\frac{1}{2}}e^{-\frac{t^2}{4}}dt
    \\
    &\geq 2\exp\left(-\tfrac{1}{4}\left(\eta+4\sqrt{\pi}\right)^2\right),
\end{align*}
and consequently
\begin{align}\label{ultimatenonsense2}
    \eta\geq \frac{1}{\sqrt{2}}A-\sqrt{H_n}-4\sqrt{\pi}.
\end{align}
Combining (\ref{ultimatenonsense1}) and (\ref{ultimatenonsense2}), we have
\begin{align}\label{XX}
    \nu_{x,0\,|\,-1}(\bar B_{-1}(z,A))\geq\Phi\left(\tfrac{1}{\sqrt{2}}A-\Lambda-\sqrt{H_n}-4\sqrt{\pi}\right):=\Phi_0\quad\text{ for all }\quad x\in B_0(x_0,\Lambda).
\end{align}
Finally, letting $x'\in \bar B_{0}(x_0,\Lambda)$ be the point where $\tilde H(\cdot,0)$ attains its minimum on $\bar B_0(x_0,\Lambda)$, we have
\begin{align*}
    \min_{\bar B_0(x_0,\Lambda)} H(\cdot,0)&\geq\min_{\bar B_0(x_0,\Lambda)} \tilde H(\cdot,0) =\tilde H(x',0)
    \\
    &=\int_{\bar B_{-1}(z,A)}H(\cdot,-1)d\nu_{x',0\,|\,-1}
    \\
    &\geq \nu_{x',0\,|\,-1}(\bar B_{-1}(z,A))\cdot\min_{\bar B_{-1}(z,A)} H(\cdot,-1)
    \\
    &\geq \Phi_0\cdot \min_{\bar B_{-1}(z,A)} H(\cdot,-1).
\end{align*}
This finishes the proof of the Theorem.
\end{proof}

\section{Local monotonicity and Nash entropy}

In this section, we prove Theorem \ref{prop: almost mono}, Theorem \ref{thm: Nash depends on nu}, and Corollary \ref{BWang}. The proof of Theorem \ref{prop: almost mono} is inspired by \cite{W18} and \cite{TZ21}. The idea is to evolve the minimizer of the local $\mu$-functional at $t=0$ using conjugate heat flow, and we may then compare the local $\mu$-functional at different time-slices using the differential Harnack inequality in \cite{W18}. The error terms come from the concentration estimate of the heat kernel measure near an $H_n$-center \cite[Proposition 3.14]{Bam20a}. The proof of Theorem \ref{thm: Nash depends on nu} is similar to that of Theorem \ref{prop: almost mono}. Given Theorem \ref{prop: almost mono}, the proof of Corollary \ref{BWang} is reduced to estimating the distance between $(x_0,0)$ and an $H_n$-center of $(x_0,T)$.

\begin{proof}[\textbf{Proof of Theorem \ref{prop: almost mono}}]
By parabolic rescaling, we may assume that $r=1.$
For simplicity, we write $$B=B_{-1}\left(z,\tfrac{7}{4}A\sqrt{H_n}\right),\quad
B'=B_{-1}\left(z,2A\sqrt{H_n}\right),$$
where $(z,-1)$ is an $H_n$-center of $(x_0,0).$
  Let $u_0$ be the minimizer of $\mu:=\mu\left(B_0\left(x_0,A\sqrt{H_n}\right),g_0,\tau\right).$ 
Following \cite{W18}, we let $u_t$ be the solution to the conjugate heat equation $\Box_t^*u_t=0$ on $M\times [-1,0]$ with initial data being $u_0$ at $t=0$.
Defining $\tau_t:=\tau-t$, by \cite[Theorem 4.2]{W18}, we have
\[
    \Big\{
        \tau_t \left( -2\tfrac{\Delta u_t}{u_t} + |\nabla\log u_t|^2+R \right)
        -\log u  - n - \mu - \tfrac{n}{2}\log(4\pi\tau_t)
    \Big\}u_t \le 0\quad\text{ on }\quad M\times [-1,0).
\]
As before, $d\mu_t:=u_t\, dg_t$ is a probability measure for $t\in [-1,0].$
\\

\noindent\textbf{Claim:}
If $A\geq 16,$ then we have
\[
    \mu_{-1}(M\setminus B)
    \le 2e^{-\frac{A^2}{20}}
    <\tfrac{1}{2}.
\]

\begin{proof}[Proof of the Claim.]

For any $x\in B_0\left(x_0,A\sqrt{H_n}\right),$ let $(z_x,-1)$ be an $H_n$-center of $(x,0).$ Then
\begin{align*}
    &\dist_{-1}(z_x,z)
     = \dist_{W_1}^{g_{-1}}(\delta_{z_x},\delta_{z})\\
    \le&\  \dist_{W_1}^{g_{-1}}(\delta_{z_x},\nu_{x,0\,|\,-1})
    +\dist_{W_1}^{g_{-1}}(\nu_{x,0\,|\,-1},\nu_{x_0,0\,|\,-1})
    +\dist_{W_1}^{g_{-1}}(\delta_{z},\nu_{x_0,0\,|\,-1})\\
    \leq&\  2\sqrt{H_n}+\dist^{g_0}_{W_1}(\nu_{x,0\,|\,0},\nu_{x_0,0\,|\,0})
    \\
    \leq&\ 2\sqrt{H_n}+A\sqrt{H_n}\\
    \leq &\ \tfrac{5}{4}A\sqrt{H_n},
\end{align*}
where we used the monotonicity of the $W_1$-Wassernstein distance (c.f. \cite[Lemma 2.7]{Bam20a}). 
Then  
$
    B_{-1}\left(z_x,\frac{1}{2}A\sqrt{H_n}\right)\subseteq B
$
for any $x\in B_0\left(x_0,A\sqrt{H_n}\right).$ By \cite[Proposition 3.14]{Bam20a}, for any $x\in B_0\left(x_0,A\sqrt{H_n}\right)$, if $A\ge 16$, then we have
\[
    \nu_{x,0\,|\,-1}(M\setminus B)
    \le \nu_{x,0\,|\,-1}\left((M\setminus B_{-1}\left(z_x,\tfrac{1}{2}A\sqrt{H_n}\right)\right)
    \le
    2e^{-\frac{A^2}{20}}.
\]
 By the standard semi-group property, we have
\[
    \mu_{-1}(M\setminus B)
    = \int_{B_0\left(x_0,A\sqrt{H_n}\right)} \nu_{x,0\,|\,-1}(M\setminus B)\cdot u_0(x)\, dg_0(x)
    \le 
    2e^{-\frac{A^2}{20}}.
\]
\end{proof}

\def \tu {\tilde{u}}

Let $\eta$ be a smooth cutoff function supported in $B'$ such that $\eta|_B=1$, $0\le \eta\le 1$, and $|\nabla \eta|^2\le \frac{160}{A^2H_n}\eta.$
Define
\[
    \alpha := \int_M \eta^2u_{-1} \, dg_{-1}= \int_M \eta^2\, d\mu_{-1}.
\]
By the Claim, we have $$\tfrac{1}{2}\leq 1-2e^{-\frac{A^2}{20}}\le \alpha \le 1.$$ Let $\tu=\eta^2u_{-1}/\alpha.$ Then $\tu\in C_0^\infty(B')$ and $\int \tu\, dg_{-1}=1.$
In the following, we omit the measure $dg_{-1}$ and write $u=u_{-1}$ when there is no ambiguity.
Integrating by parts, we have
\begin{align*}
    \int |\nabla \log \tu|^2\tu\, &= \alpha^{-1} \int |\nabla \log u|^2\eta^2 u + 2\langle\nabla \eta^2, \nabla u\rangle + 4|\nabla \eta|^2 u\\
    &= \alpha^{-1} \int \big(|\nabla \log u|^2u -2\Delta u\big) \eta^2
    + 4|\nabla\eta|^2u.
\end{align*}
Then
\begin{align}\label{localmuestimate}
    \mu(B',g_{-1},1+\tau)\le 
    &\ \int \left(\tau_{-1}(|\nabla \log \tu|^2+R)\tu - \log\tu \cdot \tu \right) dg_{-1} - n - \frac{n}{2}\log(4\pi \tau_{-1})\\\nonumber
    = &\  \alpha^{-1}\int 
    \Big\{\tau_{-1}\big(|\nabla \log u|^2 -2\tfrac{\Delta u}{u}
    +R \big) - \log u 
    - n - \tfrac{n}{2}\log(4\pi\tau_{-1})\Big\}\eta^2 u\\\nonumber
    &\quad + \alpha^{-1}\int 4 |\nabla\eta|^2u
    - \int \log \tfrac{\eta^2}{\alpha} \cdot \tfrac{\eta^2}{\alpha} u\\\nonumber
    \le &\ \mu + \frac{C_n}{\alpha A^2} \mu_{-1}(M\setminus B)
    \le  \mu + \frac{C_n}{A^2}e^{-\frac{A^2}{20}},
\end{align}
where we have implemented the following consequence of Jensen's inequality applied to the convex function $t\mapsto t\log t$ and the probability measure $u_{-1}\,dg_{-1}$
\[
    - \int \log \tfrac{\eta^2}{\alpha} \cdot \tfrac{\eta^2}{\alpha} u
    \le - \int \tfrac{\eta^2}{\alpha} u\cdot \log \int \tfrac{\eta^2}{\alpha} u =0.
\]
\end{proof}

\begin{proof}[\textbf{Proof of Theorem \ref{thm: Nash depends on nu}}]
The proof of Theorem \ref{thm: Nash depends on nu} is similar to the above proof. We shall again assume $r=1$ by parabolic rescaling. We define
\begin{gather*}
    B=B_{-1}\left(z,A\sqrt{H_n}\right),\quad B'=B_{-1}\left(z,2A\sqrt{H_n}\right),\\ u_t=(-4\pi t)^{-\frac{n}{2}}e^{-f_t}=: K(x_0, 0\,|\,\cdot,t)\quad\text{ for }\quad t\in[-1,0),
\end{gather*} 
where $(z,-1)$ is an $H_n$-center of $(x_0,0)$. As before, let $\eta$ be the cut-off function defined on $(M,g_{-1})$, satisfying $\eta|_B=1$, $\eta|_{M\setminus B'}=0$, $0\leq\eta\leq 1$, and $|\nabla\eta|^2\leq \frac{10}{A^2H_n}\eta$, and let
$$\alpha:=\int_M u_{-1}\eta^2dg_{-1}.$$
We shall then use $$\tilde u:=\alpha^{-1}u_{-1}\eta^2$$ as a test function to estimate $\mu(g_{-1},B',1)$. 

First of all, \cite[Proposition 3.14]{Bam20a} implies that
\begin{align}\label{unitintegralcoefficient}
   \tfrac{1}{2}\leq 1-2e^{-\frac{A^2}{20}}\leq\nu_{x_0,0\,|\,-1}(B) \leq\alpha\leq \nu_{x_0,0\,|\,-1}(M)=1.
\end{align}
Then, we may follow the same computation as in (\ref{localmuestimate}). Since this computation is on the fixed time-slice at $t=-1$, we shall omit the subindex $-1$ and the measure notation $dg_{-1}$ when there is no ambiguity. 
\begin{align}\label{localmuestimate1}
    \mu(B',g_{-1},1)\le 
    &\ \int \left((|\nabla \log \tilde u|^2+R)\tilde u - \log\tilde u \cdot \tilde u \right) dg_{-1} - n - \frac{n}{2}\log(4\pi)\\\nonumber
    = &\  \alpha^{-1}\int 
    \Big\{\big(|\nabla \log u|^2 -2\tfrac{\Delta u}{u}
    +R \big) - \log u 
    - n - \tfrac{n}{2}\log(4\pi)\Big\}\eta^2 u\\\nonumber
    &\quad + \alpha^{-1}\int 4 |\nabla\eta|^2u
    - \int \log \tfrac{\eta^2}{\alpha} \cdot \tfrac{\eta^2}{\alpha} u
    \\\nonumber
    =&\ \alpha^{-1}\int \left(2\Delta f-|\nabla f|^2+R+f-n\right)\eta^2u + \alpha^{-1}\int 4 |\nabla\eta|^2u
    - \int \log \tfrac{\eta^2}{\alpha} \cdot \tfrac{\eta^2}{\alpha} u
    \\\nonumber
    =&\ \alpha^{-1}\int(|\nabla f|^2+R+f-n)\eta^2u -\alpha^{-1}\int 4\langle \nabla f,\nabla\eta\rangle \eta u
        \\\nonumber
     &\quad + \alpha^{-1}\int 4 |\nabla\eta|^2u
    - \int \log \tfrac{\eta^2}{\alpha} \cdot \tfrac{\eta^2}{\alpha} u.
\end{align}
The last two terms are easily dealt with using the same argument as in the proof of the previous theorem, we shall consider the first two terms. 

First of all, we observe that the first term can be split into two terms.
\begin{align}\label{nonsenselocal1}
    &\alpha^{-1}\int(|\nabla f|^2+R+f-n)\eta^2u
    \\\nonumber
    =&\ \alpha^{-1}\int(|\nabla f|^2+R)\eta^2u+\alpha^{-1}\int\left(f-\tfrac{n}{2}\right)\eta^2u -\frac{n}{2}
    \\\nonumber
    =&:\ \text{I}+\text{II} -\frac{n}{2}.
\end{align}
Applying \cite[Proposition 5.13]{Bam20a} (see also \cite[Proposition 3.3]{CMZ21a} for the proof under our current condition), we may estimate the term I as follows.
\begin{align}\label{nonsenselocal2}
    \text{I} & = \alpha^{-1}\int(|\nabla f|^2+R)\eta^2u
    \\\nonumber
    &= \alpha^{-1}\int(|\nabla f|^2+R-\Rmin)\eta^2u+\Rmin
    \\\nonumber
    &\leq \alpha^{-1}\int(|\nabla f|^2+R-\Rmin) u+\Rmin
        \\\nonumber
    &= \alpha^{-1}\int(|\nabla f|^2+R) u -(\alpha^{-1}-1)\Rmin
    \\\nonumber
    &\leq \alpha^{-1}\cdot\frac{n}{2}-(\alpha^{-1}-1)\Rmin.
\end{align}
We may apply \cite[Proposition 5.13]{Bam20a} again as well as the definition of the Nash entropy (\ref{def_nash}) to estimate the term II as follows. For the sake of notational simplicity, we have defined $\N:=\N_{x_0,0}(1)$.
\begin{align}\label{nonsenselocal3}
    \text{II} &=\alpha^{-1}\int\left(f-\tfrac{n}{2}-\N\right)\eta^2u +\N
    \\\nonumber
    &=\alpha^{-1}\int\left(f-\tfrac{n}{2}-\N\right)(\eta^2-1)u +\N
    \\\nonumber
    &\leq \alpha^{-1}\left(\int\left(f-\tfrac{n}{2}-\N\right)^2u\right)^{\frac{1}{2}}\left(\int\left(\eta^2-1\right)^2u\right)^{\frac{1}{2}}+\N
    \\\nonumber
    &\leq \alpha^{-1}(\nu_{-1}(M\setminus B))^{\frac{1}{2}}(n-2\Rmin)^{\frac{1}{2}}+\N
    \\\nonumber
    &\leq \sqrt{2}\alpha^{-1}\cdot(n-2\Rmin)^{\frac{1}{2}}\cdot e^{-\frac{A^2}{40}}+\N,
\end{align}
where we have applied (\ref{unitintegralcoefficient}). Combining (\ref{nonsenselocal1}),  (\ref{nonsenselocal2}), and (\ref{nonsenselocal3}), we have
\begin{align}\label{nonsenselocal4}
    &\alpha^{-1}\int(|\nabla f|^2+R+f-n)\eta^2u
    \\\nonumber
    \leq &\ \N+ \tfrac{1}{2}(\alpha^{-1}-1)\cdot(n-2\Rmin)+\sqrt{2}\alpha^{-1}e^{-\frac{A^2}{40}}\cdot(n-2\Rmin)^{\frac{1}{2}}
    \\\nonumber
    \leq &\ \N+4\left(e^{-\frac{A^2}{20}}\cdot(n-2\Rmin)+e^{-\frac{A^2}{40}}\cdot(n-2\Rmin)^{\frac{1}{2}}\right),
\end{align}
where we have applied (\ref{unitintegralcoefficient}).

Next, we apply the Cauchy-Schwarz inequality to estimate the second term on the right-hand-side of (\ref{localmuestimate1}).
\begin{align}\label{nonsenselocal5}
    -\alpha^{-1}\int 4\langle \nabla f,\nabla\eta\rangle \eta u&\leq 4\alpha^{-1}\left(\int|\nabla f|^2u\right)^{\frac{1}{2}}\left(\int|\nabla \eta|^2\eta^2 u\right)^{\frac{1}{2}}
    \\\nonumber
    &\leq 4\alpha^{-1}\left(\int(|\nabla f|^2+R)u-\Rmin\right)^{\frac{1}{2}}\left(\nu_{-1}(M\setminus B)\right)^{\frac{1}{2}}
    \\\nonumber
    &\leq 4\alpha^{-1}e^{-\frac{A^2}{40}}\cdot(n-2\Rmin)^{\frac{1}{2}}.
\end{align}

Finally, arguing as in the proof of Theorem \ref{prop: almost mono}, we have
\begin{align}\label{nonsenselocal6}
    \alpha^{-1}\int 4 |\nabla\eta|^2u\leq\alpha^{-1}\frac{C_n}{A^2}e^{-\frac{A^2}{20}},\quad - \int \log \tfrac{\eta^2}{\alpha} \cdot \tfrac{\eta^2}{\alpha} u\leq 0.
\end{align}
Combining (\ref{localmuestimate1}), (\ref{nonsenselocal4}), (\ref{nonsenselocal5}), and (\ref{nonsenselocal6}), the conclusion follows.
\end{proof}

\begin{proof}[\textbf{Proof of Corollary \ref{BWang}}]
Let us assume $T=1$. By Theorem \ref{prop: almost mono}, the proof is reduced to estimating the distance between $(x_0,0)$ and an $H_n$-center of $(x_0,1)$. Let $\psi:\mathbb R\to [0,1]$ be the cut-off function applied by \cite{W18}, satisfying $\psi|_{(-\infty,1)}\equiv1$, $\psi_{(2,\infty)}\equiv 0$, and 
\begin{align}\label{BWangcutoff}
    \psi''\ge-10\psi\quad\text{ and }\quad (\psi')^2\leq 10\psi.
\end{align}
Then, arguing as the proof of \cite[Theorem 5.4]{W18} (which is merely an application of \cite[Lemma 8.3(a)]{Per02}), we have that the function
$$\psi(x,t):=\psi\left(\frac{\dist_t(x_0,x)+A\sqrt{t}}{A+1}\right),\quad\text{ where }\quad (x,t)\in M\times[0,1],$$
satisfies the inequality
$$\Box \psi\leq\frac{10}{(A+1)^2}\psi.$$

Therefore, we may compute
$$\frac{d}{dt}\int_M\psi(\cdot,t)d\nu_{x_0,1\,|\,t}=\int_M\Box_t\psi(\cdot,t)d\nu_{x_0,1\,|\,t}\leq\frac{10}{(A+1)^2}\int_M\psi(\cdot,t)d\nu_{x_0,1\,|\,t},$$
and consequently
\begin{align}\label{distnonsense1}
    \nu_{x_0,1\,|\,0}\left(B_0(x_0,3A)\right)\geq \int_M\psi(\cdot,0)d\nu_{x_0,1\,|\,0}\ge e^{-\frac{10}{(A+1)^2}}\int_M\psi(\cdot,1)d\nu_{x_0,1\,|\,1}= e^{-\frac{10}{(A+1)^2}}.
\end{align}
Letting $(z,0)$ be an $H_n$-center of $(x_0,1)$ and $r:=\dist_0(x_0,z)$. Then, by \cite[Proposition 3.14]{Bam20a}, whenever $r>3A$, we have
\begin{align}\label{distnonsense2}
    \nu_{x_0,1\,|\,0}\left(B_0(x_0,3A)\right)&\leq \nu_{x_0,1\,|\,0}\left(M\setminus B_0(z,r-3A)\right)
    \\\nonumber
    &\leq 2\exp\left(-\frac{(r-3A-\sqrt{2H_n})_+^2}{8}\right).
\end{align}
Combining (\ref{distnonsense1}) and (\ref{distnonsense2}), we have
\begin{align*}
    r\leq 3A+\sqrt{2H_n}+\left(\frac{80}{(A+1)^2}+8\log2\right)^{\frac{1}{2}}\leq 4A,
\end{align*}
where we have applied the fact that $A\ge 1000n$.

Finally, applying Theorem \ref{prop: almost mono}, we have
\begin{align*}
    \mu(B_1(x_0,8A),g_1,\tau)&\ge \mu(B_0(z,16A),g_0,1+\tau)-\frac{C_n}{A^2}e^{-c_nA^2}
    \\
    &\ge \mu(B_0(x_0,20A),g_0,1+\tau)-\frac{C_n}{A^2}e^{-c_nA^2}.
\end{align*}
The corollary then follows.
\end{proof}

\section{Perelman's noncollapsing improving and pseudolocality theorems}

\subsection{Noncollapsing improving theorem}

In this subsection, we prove Theorem \ref{Theorem2}. The proof is almost identical to \cite{J21}, but we sketch it for the convenience of the reader. The main idea is that the integral bound on the scalar curvature guarantees that $x_0$ is not far away from any $H_n$-center at time $-r^2$. The theorem follows from the Harnack property of the Nash entropy and Theorem \ref{thm: nu ge nash}.

\begin{proof}[\textbf{Proof of Theorem \ref{Theorem2}}]
As usual, we assume $r=1$ by parabolic rescaling. Let $(z,-1)$ be an $H_n$-center of $(x_0,0)$, then Bamler's conjugate heat kernel estimate (c.f. \cite[Theorem 7.2]{Bam20a}; see \cite[Theorem 3.2]{CMZ21a} for the proof under the bounded curvature assumption) together with Theorem \ref{subsolution} implies that
\begin{eqnarray}\label{nonsense0001}
\frac{1}{(4\pi)^{\frac{n}{2}}}e^{-\ell_{x_0,0}(x_0,1)}\leq K(x_0,0\,|\,x_0,-1)\leq \frac{C}{(4\pi)^{\frac{n}{2}}}\exp\left(-\N_{x_0,0}(1)-\frac{\dist^2_{-1}(x_0,z)}{C}\right),
\end{eqnarray}
where $C$ is a numerical constant depending only on the dimension, and $\ell_{x_0,0}$ is Perelman's reduced distance based at $(x_0,0)$ which can be estimated by
\begin{eqnarray}\label{nonsense0002}
\ell_{x_0,0}(x_0,1)\leq\frac{1}{2\sqrt{1}}\int_0^1\sqrt{\tau}R_{g_{-\tau}}(x_0)d\tau\leq\frac{A}{2},
\end{eqnarray}
where we have used the static curve at $x_0$ as a test curve. Combining (\ref{nonsense0001}) and (\ref{nonsense0002}), we have
\begin{eqnarray}\label{nonsense003}
\dist_{-1}(x_0,z)&\leq& C(n,A)\left(-\N_{x_0,0}(1)+1\right)^{\frac{1}{2}}
\\\nonumber
&\leq& C(n,A)\left(-\N_{x_0,0}(2)+1\right)^{\frac{1}{2}}.
\end{eqnarray}
On the other hand, since $(z,-1)$ is an $H_n$-center of $(x_0,0)$, by (\ref{HNW1}), we have 
\begin{eqnarray}\label{nonsense0004}
\dist_{W_1}^{g_{-1}}(\nu_{x_0,0\,|\,-1},\delta_{x_0})&\leq& \dist_{W_1}^{g_{-1}}(\nu_{x_0,0\,|\,-1},\delta_{z})+\dist_{W_1}^{g_{-1}}(\delta_{x_0},\delta_{z})
\\\nonumber
&\leq& \sqrt{H_n}+\dist_{-1}(x_0,z)
\\\nonumber
&\leq& C(n)+C(n,A)\left(-\N_{x_0,0}(2)+1\right)^{\frac{1}{2}},
\end{eqnarray}
where we have also applied (\ref{nonsense003}). Applying Theorem \ref{Harnack} with $(x_0,-1)\to (x_1,t_1)$, $(x_0,0)\to (x_2,t_2)$, $-1\to t^*$, $-2\to s$, and taking (\ref{nonsense0004}) and the fact that $R_{g_{-1}}\geq -\frac{n}{2}$ (obtained in the same ways as (\ref{Rlowerbound})) in to account, we have
\begin{align*}
    \N_{x_0,-1}(1)&\leq \N_{x_0,0}(2)+\left(\frac{n}{2(-1-(-2))}-(-\frac{n}{2})\right)^{\frac{1}{2}}\dist^{g_{-1}}_{W_1}(\nu_{x_0,0\,|\,-1},\delta_{x_0})+\frac{n}{2}\log\left(\frac{0-(-2)}{-1-(-2)}\right)
    \\
    &\leq \N_{x_0,0}(2)+C(n)+C(n,A)\left(-\N_{x_0,0}(2)+1\right)^{\frac{1}{2}}
    \\
    &\leq \N_{x_0,0}(2)+C(n)+\frac{1}{2}\left(-\N_{x_0,0}(2)+1\right)+C(n,A)
    \\
    &\leq \frac{1}{2}\N_{x_0,0}(2)+C(n,A).
\end{align*}
Hence, applying \cite[Theorem 8.1]{Bam20a} in the same way as in (\ref{X}), we have
\begin{align*}
    \N_{x_0,0}(2)\geq 2\N_{x_0,-1}(1)-C(n,A)\geq -C(n,A),
\end{align*}
and the rest of the proof is but a straightforward application of Theorem \ref{thm: nu ge nash}.
\end{proof}

\subsection{Pseudolocality theorem}

In this subsection, we present the proof of Theorem \ref{BWang2} using our local monotonicity theorem. The proof goes by contradiction. We show that under the assumption of Theorem \ref{BWang2}, a point at which (\ref{pseudo2}) fails has very small Nash entropy, and this is a contradiction. To this end, we need the following preparatory result due to Bamler.

\begin{Lemma}[{\cite[Theorem 10.3]{Bam20a}}]\label{lem:regularity}
For any $\varepsilon>0$ there is $\delta=\delta(n,\varepsilon)>0$, such that the following holds. Let $(x,t)\in M\times I$ and $r>0$ satisfy $[t-r^2,t]\subset I$ and $\N_{x,t}(r^2)\geq-\delta$. Then we have
\begin{gather}
    |{\Rm}|\leq \varepsilon r^{-2}\quad \text{ on }\quad B_t(x,\varepsilon^{-1}r)\times([t-(1-\varepsilon)r^2,\varepsilon^{-1}r^2]\cap I),\label{regularity1}
    \\
    \inf_{\rho\in(0,\varepsilon^{-1}r)}\rho^{-n}|B_t(x,\rho)|_t\geq (1-\varepsilon)\omega_n,\label{regularity2}
\end{gather}
where $\omega_n$ is the volume of the $n$-dimensional unit disk in the Euclidean space.
\end{Lemma}

\begin{proof}
(\ref{regularity1}) is the same as the statement of \cite[Theorem 10.3]{Bam20a}, and we show that (\ref{regularity2}) is a straightforward consequence of (\ref{regularity1}).

Without loss of generality, we assume $t=0$ and $r=1$. Let us consider a sequence of counter examples $(M_i^n,g_{i,t},x_i)_{t\in[-1,0]}$ with $\N_{x_i,0(1)}=-\delta_i\nearrow 0$, but
\begin{align}\label{regularity3}
\inf_{\rho\in(0,\varepsilon^{-1})}\rho^{-n}|B_{g_{i,0}}(x_i,\rho)|_0\le (1-\varepsilon)\omega_n    
\end{align}
for some fixed $\varepsilon>0$.

By \eqref{regularity1}, we can find a sequence $\varepsilon_i\searrow0$, such that
$$|{\Rm_{g_i}}|\leq \varepsilon_i \quad \text{ on }\quad B_{g_{i,0}}(x_i,\varepsilon_i^{-1})\times[-1+\varepsilon_i,0]$$
and hence we may extract a subsequence from $(M_i,g_{i,t},x_i)$ and obtain a limit Ricci flow. This limit must be the Euclidean space, and this contradicts (\ref{regularity3}).
\end{proof}

Since, to apply Theorem \ref{thm: Nash depends on nu}, we need a global scalar curvature lower bound on the initial slice, which is not available in the assumption of Theorem \ref{BWang2}, we combine Theorem \ref{thm: Nash depends on nu} and Theorem \ref{prop: almost mono} to obtain the following result for the convenience of application.

\begin{Proposition}\label{Prop: Nash bounded by nu}
Assume that $[-r^2,0]\subseteq I$. Then, for any $x_0\in M$, any $H_n$-center $(z,-r^2)$ of $(x_0,0)$, and any $A\ge \underline A(n)$, we have
\begin{align*}
     \mu\left(B_{{-r^2}}\left(z,Ar\right), g_{-r^2},r^2\right)
    \le \N_{x_0,0}\left(\tfrac{1}{2}r^2\right)
    +  C_ne^{-c_nA^2},
\end{align*}
where $\underline A(n)$, $c_n$, and $C_n$ are dimensional constants.
\end{Proposition}

To prove the above proposition we need the following observation, namely, that an $H_n$-center of an $H_n$-center is almost an $H_n$-center; this result is interesting in its own right, since the conclusion is independent of $t_2\in(t_1,t_3)$.

\begin{Lemma}\label{Hngeodesic}
Let $x_1$, $x_2$, $x_3$, $z\in M$ and $t_1$, $t_2$, $t_3\in I$ such that $t_1<t_2<t_3$. Furthermore, assume that $(x_1,t_1)$ is an $H_n$-center of $(x_2,t_2)$, that $(x_2,t_2)$ is an $H_n$-center of $(x_3,t_3)$, and that $(z,t_1)$ is an $H_n$-center of $(x_3,t_3)$. Then we have
$$\dist_{t_1}(x_1,z)\leq C_n\sqrt{t_3-t_1},$$
where $C_n$ is a dimensional constant.
\end{Lemma}

\begin{proof}
Letting $A>\sqrt{2H_n}+4\sqrt{2\pi}$ and $\Lambda>0$, we may argue as the proof of (\ref{XX}) and obtain 
\begin{align*}
    \nu_{x,t_2\,|\,t_1}\left(B_{t_1}\left(x_1,A\sqrt{t_2-t_1}\right)\right)\geq\Phi\left(\tfrac{1}{\sqrt{2}}A-\Lambda-\sqrt{H_n}-4\sqrt{\pi}\right)
\end{align*}
for all $x\in B_{t_2}\left(x_2,\Lambda\sqrt{t_2-t_1}\right)$, where $\Phi$ is the function in the statement of Theorem \ref{HeatHarnack}.

Taking $\Lambda:=\sqrt{2H_n\frac{t_3-t_2}{t_2-t_1}}$, and applying the fact (c.f. \cite[Proposition 3.13]{Bam20a}) that $$\nu_{x_3,t_3\,|\,t_2}\left(B_{t_2}\left(x_2,\sqrt{2H_n(t_3-t_2)}\right)\right)\geq\frac{1}{2},$$ we have
\begin{align*}
    \nu_{x_3,t_3\,|\,t_1}\left(B_{t_1}\left(x_1,A\sqrt{t_2-t_1}\right)\right)&\geq\int_{B_{t_2}\left(x_2,\sqrt{2H_n(t_3-t_2)}\right)}\nu_{\cdot,t_2\,|\,t_1}\left(B_{t_1}\left(x_1,A\sqrt{t_2-t_1}\right)\right)d\nu_{x_3,t_3\,|\,t_2}
    \\
    &\ge \frac{1}{2}\Phi\left(\tfrac{1}{\sqrt{2}}A-\sqrt{2H_n\tfrac{t_3-t_2}{t_2-t_1}}-\sqrt{H_n}-4\sqrt{\pi}\right).
\end{align*}
Finally, taking $A:=\sqrt{4H_n\frac{t_3-t_1}{t_2-t_1}}$, we have
\begin{align*}
    \nu_{x_3,t_3\,|\,t_1}\left(B_{t_1}\left(x_1,\sqrt{4H_n(t_3-t_1)}\right)\right)&\ge \frac{1}{2}\Phi\left(\sqrt{2H_n\tfrac{t_3-t_1}{t_2-t_1}}-\sqrt{2H_n\tfrac{t_3-t_2}{t_2-t_1}}-\sqrt{H_n}-4\sqrt{\pi}\right)
    \\
    &=\Phi\left(\frac{2H_n}{\sqrt{2H_n\tfrac{t_3-t_1}{t_2-t_1}}+\sqrt{2H_n\tfrac{t_3-t_2}{t_2-t_1}}}-\sqrt{H_n}-4\sqrt{\pi}\right)
    \\
    &\geq \Phi\left(-\sqrt{H_n}-4\sqrt{\pi}\right)=:c_n.
\end{align*}
It then follows from \cite[Proposition 3.13]{Bam20a} that
$$\dist_{t_1}(x_1,z)\leq (c_n^{-\frac{1}{2}}+2)\sqrt{H_n(t_3-t_1)}.$$
\end{proof}

\begin{proof}[Proof of Proposition \ref{Prop: Nash bounded by nu}]
By parabolic scaling, let us assume $r=1$. Let $(y,-1/2)$ be an $H_n$-center of $(x_0,0)$, and $(z',-1)$ be an $H_n$-center of $(y,-1/2)$. Then, by Lemma \ref{Hngeodesic}, we have
\begin{align}\label{penultnonsense1}
    \dist_{-1}(z',z)\leq C_n,
\end{align}
where $(z,-1)$ is an $H_n$-center of $(x_0,0)$.

By the maximum principle, we have $R_{g_{-1/2}}\geq -n$. Hence, we may apply Theorem \ref{thm: Nash depends on nu} to obtain
\begin{align}\label{penultnonsense2}
    \N_{x_0,0}(1/2)\geq \mu\left(B_{-1/2}\left(y,\tfrac{1}{4}A\right),g_{-1/2},1/2\right)-C_ne^{-c_nA^2},
\end{align}
if $A\geq\underline A(n)$, where $c_n$ and $C_n$ are both dimensional constants. Applying Theorem \ref{prop: almost mono} and (\ref{penultnonsense1}), we have
\begin{align}\label{penultnonsense3}
    \mu\left(B_{-1/2}\left(y,\tfrac{1}{4}A\right),g_{-1/2},1/2\right)&\geq \mu\left(B_{-1}\left(z',\tfrac{1}{2}A\right),g_{-1},1\right)-C_ne^{-c_nA^2}
    \\\nonumber
    &\geq \mu\left(B_{-1}\left(z,A\right),g_{-1},1\right)-C_ne^{-c_nA^2},
\end{align}
if $A\ge\underline A(n)$. The proposition then follows from (\ref{penultnonsense1}) and (\ref{penultnonsense3}).
\end{proof}

\begin{proof}[\textbf{Proof of Theorem \ref{BWang2}}]
Let us first prove (\ref{pseudo2}) by contradiction, since (\ref{pseudo3}) and (\ref{pseudo4}) follow from it. Fixing an $\alpha\in(0,\frac{1}{100n})$, we shall show that under the assumption of (\ref{pseudo1}), if (\ref{pseudo2}) fails, then there is a contradiction arising when $\delta$ is small enough.

Let $\bar t\in(0,T]$ be the first time such that the following statement fails
$$|{\Rm_{g_t}}|(x)< \alpha t^{-1}\quad \text{ for all }\quad t\in(0,\bar t\, ]\ \text{ and }\ x\in \bar B_t\left(x_0,\alpha^{-1}\sqrt{t}\right).$$
Since the Ricci flow is smooth, we must have $\bar t>0$. By the contradictory assumption, we also have $\bar t<T$. Henceforth we shall assume $\bar t=1$ by a parabolic scaling. Then, by the definition of $\bar t$, we have
\begin{align*}
    |{\Rm_{g_t}}|(x)\leq\frac{\alpha}{t}\quad\text{ for all }\quad t\in(0,1]\ \text{ and }\ x\in B_t\left(x_0,\alpha^{-1}\sqrt{t}\right),
\end{align*}
and there exists $y\in \bar B_1(x_0,\alpha^{-1})$, such that 
\begin{align}\label{penultnonsense4}
    |{\Rm_{g_1}}|(y)=\alpha.
\end{align}

Let $z$ be an $H_n$-center of $y$. Then, we may apply the same argument as in the proof of Corollary \ref{BWang} by using the cutoff function of the form
$$\psi(x,t):=\psi\left(\frac{\dist_t(x_0,x)+2\alpha^{-1}\sqrt{t}}{10\alpha^{-1}}\right),$$
where $\psi$ is the function defined in (\ref{BWangcutoff}). This leads to 
\begin{align*}
    \dist_0(z,x_0)\leq C(n,\alpha).
\end{align*}

Now, if $\delta$ is taken to be small enough, such that $\delta^{-1}\gg C(n,\alpha)$, then we may apply Proposition \ref{Prop: Nash bounded by nu} to obtain
\begin{align}\label{penultnonsense5}.
    \N_{y,1}(1/2)&\geq \mu\left(B_0\left(z,\tfrac{1}{2}\delta^{-1}\right),g_0,1\right) -C_ne^{-c_n\delta^{-2}}
    \\\nonumber
    &\ge \mu\left(B_0\left(x_0,\delta^{-1}\right),g_0,1\right) -C_ne^{-c_n\delta^{-2}}
    \\\nonumber
    &\ge -\delta^2-C_ne^{-c_n\delta^{-2}},
\end{align}
where we have applied the assumption (\ref{pseudo1}). Obviously, by Lemma \ref{lem:regularity}, (\ref{penultnonsense5}) contradicts (\ref{penultnonsense4}) when $\delta$ is small enough; this proves (\ref{pseudo2}).

Next, we prove (\ref{pseudo3}) assuming (\ref{pseudo2}), and (\ref{pseudo4}) is a consequence of (\ref{pseudo2}) and (\ref{pseudo3}) 
as a well-known result proved by Cheeger-Gromov-Taylor \cite{CGT82}.
In fact, fixing any $(y,t)\in M\times(0,T]$ such that $\dist_t(x_0,y)\leq\alpha^{-1}\sqrt{t}$. Let $(z,0)$ be an $H_n$-center of $(y,t)$. Given the curvature bound (\ref{pseudo2}), we may apply the same argument as above to obtain 
\begin{align*}
    \dist_0(x_0,z)\leq C(n,\alpha)\sqrt{t}.
\end{align*}
If $\delta$ is taken to be small enough such that $\delta^{-1}\gg C(n,\alpha)$, then we argue in the same way as (\ref{penultnonsense5}) to obtain
\begin{align*}
    \N_{y,t}(\tfrac{1}{2}t)&\geq \mu\left(B_0\left(z,\tfrac{1}{2}\delta^{-1}\sqrt{t}\right),g_0,t\right) -C_ne^{-c_n\delta^{-2}}
    \\\nonumber
    &\ge \mu\left(B_0\left(x_0,\delta^{-1}\sqrt{t}\right),g_0,t\right) -C_ne^{-c_n\delta^{-2}}
    \\\nonumber
    &\ge -\delta^2-C_ne^{-c_n\delta^{-2}},
\end{align*}
where we have also applied (\ref{pseudo1}). (\ref{pseudo3}) then follows from Lemma \ref{lem:regularity}.
\end{proof}

\section{Application to ancient Ricci flows}

In this section, we provide a short proof of Theorem \ref{AncientSobolev} using Theorem \ref{thm: nu ge nash}.
\begin{proof}[\textbf{Proof of Theorem \ref{AncientSobolev}}]
Let us first of all show that
\begin{equation}\label{x1}
    \inf_{t\leq 0}\nu(g_t)\geq \mu_\infty.
\end{equation}
Fixing an arbitrary $t\in (-\infty,0]$ and $x\in M$, we have, by \cite[Proposition 4.6]{MZ21}
$$\lim_{\tau\to\infty}\N_{x,t}(\tau)=\lim_{\tau\to\infty}\N_{x_0,t_0}(\tau)=\mu_\infty.$$
Hence, by the monotonicity of the Nash entropy, we have
$$\N_{x,t}(r^2)\geq\mu_\infty\quad\text{ for all }\quad r>0.$$
Applying Theorem \ref{thm: nu ge nash} at $(x,t)$ with $A=\tau=\epsilon>0$, we have
\begin{align*}
    \nu\left(B_t(x,\epsilon r),g_t,\epsilon r^2\right)&\geq \N_{x,t}(r^2)-\sqrt{n}\epsilon-\frac{n}{2}\epsilon-\frac{n}{2}\log(1+\epsilon)
    \\
    &\geq \mu_\infty -\sqrt{n}\epsilon-\frac{n}{2}\epsilon-\frac{n}{2}\log(1+\epsilon)\quad\text{ for all }\quad r>0.
\end{align*}
By first taking $r\to\infty$ and then taking $\epsilon\to 0$, we have
$$\nu(g_t)\geq\mu_\infty.$$
Since $t\in(-\infty,0]$ is arbitrary, we have proved (\ref{x1}).

Next, we prove 
\begin{equation}\label{x2}
    \inf_{t\leq 0}\nu(g_t)\leq \mu_\infty.
\end{equation}
By the assumption of the theorem, for any $\epsilon >0$, we can find $\tau<\infty$, such that 
$$\N_{x_0,t_0}(\tau)\leq \mu_\infty+\epsilon.$$
Consequently we have
$$\mu_\infty+\epsilon\geq\N_{x_0,t_0}(\tau)\geq \W_{x_0,t_0}(\tau) \geq \nu(g_{t_0-\tau})\geq \inf_{t\leq 0}\nu(g_t).$$
Since $\epsilon$ is arbitrary, (\ref{x2}) follows immediately.
\end{proof}

\bibliographystyle{amsalpha}

\newcommand{\etalchar}[1]{$^{#1}$}
\providecommand{\bysame}{\leavevmode\hbox to3em{\hrulefill}\thinspace}
\providecommand{\MR}{\relax\ifhmode\unskip\space\fi MR }
\providecommand{\MRhref}[2]{%
  \href{http://www.ams.org/mathscinet-getitem?mr=#1}{#2}
}
\providecommand{\href}[2]{#2}

\bigskip
\bigskip


\noindent Department of Mathematics, University of California, San Diego, CA 92093, USA
\\ E-mail address: \verb"pachan@ucsd.edu "
\\

\noindent Department of Mathematics, University of California, San Diego, CA 92093, USA
\\ E-mail address: \verb"zim022@ucsd.edu"
\\

\noindent School of Mathematics, University of Minnesota, Twin Cities, MN 55414, USA
\\ E-mail address: \verb"zhan7298@umn.edu"

\end{document}